\documentclass[a4paper, oneside,11 pt]{article}

\usepackage[latin1]{inputenc}
\usepackage[T1]{fontenc}
\usepackage[english]{babel}

\usepackage[all]{xy}
\usepackage{verbatim}
\usepackage{amsmath}
\usepackage{amsthm}
\usepackage{graphicx}
\usepackage{amssymb}
\usepackage{epstopdf}
\usepackage{url}
\usepackage{amsfonts}
\usepackage{todonotes}
\usepackage{fullpage}

\newcommand{\pt}{\partial}

  \newcommand{\M}{{\mathcal M}}
\newcommand{\br}{\mathbb{R}}
\newcommand{\bz}{\mathbb{Z}}

\newcommand{\bt}{\mathbb{T}}

\newcommand{\e}{\varepsilon}

\renewcommand{\)}{\right)}

\newcommand{\bb}{\mathbb}

\newtheorem{thm}{Theorem}
\newtheorem{lem}[thm]{Lemma}
\newtheorem{cor}[thm]{Corollary}
\newtheorem{prop}[thm]{Proposition}
\newtheorem{defi}[thm]{Definition}
\newtheorem{remark}[thm]{Remark}

\def\be{\begin{equation}}
\def\ee{\end{equation}}
\def\bea{\begin{eqnarray}}
\def\eea{\end{eqnarray}}

\newcommand{\f}{{\mathcal{F}}}
\newcommand{\E}{{\mathcal E}}
\newcommand{\U}{{\mathcal U}}
\newcommand{\rrho}{{\varrho}}

\numberwithin{thm}{section}
\numberwithin{equation}{section}

\newcommand{\Black}{\color{black}}

\author{Daniel Han-Kwan\footnote{CNRS $\&$ \'Ecole Polytechnique, Centre de Math\'ematiques Laurent Schwartz, UMR 7640. Email: daniel.han-kwan@math.polytechnique.fr}, \, Mikaela Iacobelli\footnote{Universit\`a ``La Sapienza'', Dipartimento di Matematica ``Guido Castelnuovo'', Roma.
Email: iacobelli@mat.uniroma1.it} \footnote{
\'Ecole Polytechnique, Centre de Math\'ematiques Laurent Schwartz, UMR 7640. 
}
}

\title{The quasineutral limit of the Vlasov-Poisson equation in Wasserstein metric}

\begin{document}

\maketitle

\begin{abstract}
In this work, we study the quasineutral limit of the one-dimensional Vlasov-Poisson equation for ions with massless thermalized electrons. We prove new weak-strong stability estimates in the Wasserstein metric that allow us to extend and improve previously known convergence results. In particular, we show that given a possibly unstable analytic initial profile, the formal limit holds for sequences of measure initial data converging sufficiently fast in the Wasserstein metric to this profile. This is achieved without assuming uniform analytic regularity.
\end{abstract}


\section{Introduction}

In this paper we study a Vlasov-Poisson system, which is a model describing the dynamics of ions in a plasma, in the presence of massless thermalized electrons. We shall focus on the one-dimensional case and consider that the equations are set on the phase space $\bt \times \br$ (we will sometimes identify $\bt$ to $[-1/2,1/2)$ with periodic boundary conditions). The system, which we shall refer to as the Vlasov-Poisson system with \emph{massless electrons}, encodes the fact that 
electrons move very fast and quasi-instantaneously reach their local thermodynamic equilibrium.
It
reads as follows:
\begin{equation}
\label{vpme}
(VPME):= \left\{ \begin{array}{ccc}\pt_t f+v\cdot \pt_x f+ E\cdot \pt_v f=0,  \\
E=-U', \\
U''=e^U- \int_{\br} f\, dv=:e^U- \rho,\\
f\vert_{t=0}= f_0\ge0,\ \  \int_{\bt \times \br} f_0\,dx\,dv=1.
\end{array} \right.
\end{equation}
Here, as usual, $f(t,x,v)$ stands for the distribution function of the ions in the phase space $\bt \times \br$ at time $t \in \br^+$, while $U(t,x)$ and $E(t,x)$ represent the electric potential and field respectively, and $U'$ (resp. $U''$) denotes the first (resp. second) spatial derivative of $U$. In the Poisson equation,  the semi-linear term $e^U$ stands for the density of electrons, which therefore are assumed to follow a Maxwell-Boltzmann law. We refer for instance to \cite{HK} for a physical discussion on this system.

We are interested in the behavior of solutions to the (VPME) system in the so-called \emph{quasineutral limit}, i.e., when the Debye length of the plasma vanishes. 
Loosely speaking, the Debye length can be interpreted as the typical scale of variations of the electric potential.
It turns out that it is always very small compared to the typical observation length, so that the quasineutral limit is relevant from the physical point of view. As a result, the approximation which consists in considering a Debye length equal to zero is widely used in plasma physics, see for instance \cite{Chen}. This leads to the study of the limit as $\e\to0$ of the scaled system:
\begin{equation}
\label{vpme-quasi}
(VPME)_\e:= \left\{ \begin{array}{ccc}\pt_t f_\e+v\cdot \pt_x f_\e+ E_\e\cdot \pt_v f_\e=0,  \\
E_\e=-U_\e', \\
\e^2 U_\e''=e^{U_\e}- \int_{\br} f_\e\, dv=:e^{U_\e}- \rho_\e,\\
f_\e\vert_{t=0}=f_{0,\e}\ge0,\ \  \int_{\bt \times \br} f_{0,\e}\,dx\,dv=1.
\end{array} \right.
\end{equation}
The formal limit is obtained in a straightforward way by taking $\e=0$ (which corresponds to a Debye length equal to $0$):
\begin{equation}
\label{formal}
(KIE):= \left\{ \begin{array}{ccc}\pt_t f+v\cdot \pt_x f+ E\cdot \pt_v f=0,  \\
E=-U', \\
U= \log \rho,\\
f_0\ge0,\ \  \int_{\bt \times \br} f_0\,dx\,dv=1,
\end{array} \right.
\end{equation}
a system we shall call the \emph{kinetic isothermal Euler system}.

We point out that there are variants of the (VPME) system which are also worth studying, such as the linearized (VPME), in which semi-linear term in the Poisson equation is linearized (this turns out to be a standard approximation in plasma physics, see also \cite{HK10,HK,HK13,HKH}),
$$
U''= U +1 - \rho,
$$
and the Vlasov-Poisson system for electrons with fixed ions (the most studied model in the mathematical literature), in which the Poisson equation reads as follows
$$
U''= 1-\rho,
$$
which we shall refer to as the \emph{classical} Vlasov-Poisson system.
As a matter of fact, our results concerning the (VPME) system have analogous statements for the linearized (VPME) or the classical Vlasov-Poisson system.
We have made the choice to study the (VPME) system since the semi-linear term in the Poisson equation creates additional interesting difficulties.
As we shall mention in Remark \ref{rmk:VP}, our analysis applies as well, \emph{mutatis mutandis}, to these models,
and actually provides a stronger result in terms of the class of data that we are allowed to consider.

\bigskip

The justification of the limit $\e \to 0$ from \eqref{vpme-quasi} to \eqref{formal} 
is far from trivial. 
Indeed, this is known to be true only in few cases (see also  \cite{Br89,Gr95,HKH} for further insights):
when the sequence of initial data $f_{0,\e}$ enjoys uniform analytic regularity
with respect to the space variable
(as we shall describe later  in Section \ref{subsec:grenier}, this is just an adaptation of a work of Grenier \cite{Gr96} on the classical Vlasov-Poisson system); when $f_{0,\e}$ converge to a Dirac measure in velocity 
$f_0(x,v) = \rho_0(x) \delta_{v= v_0(x)}$ (see \cite{HK} and \cite{Br00,Mas,GSR}); and, following \cite{HKH},
when $f_{0,\e}$ converge to a homogeneous initial condition $\mu(v)$ which is symmetric with respect to some $\overline{v} \in \br$
and which is first increasing then decreasing.
Also, it is conjectured (see \cite{Gr99}) that this result should hold when the sequence of initial data $f_{0,\e}$ converges to some $f_0$ such that, for all $x \in \bt$, $v \mapsto f_{0}(x,v)$ satisfies a stability condition \emph{a la Penrose} \cite{Pen} (typically when $v \mapsto f_{0}(x,v)$ is increasing then decreasing). On the other hand, the limit is known to be false in general, as we will explain later.

\bigskip

In this work, we shall study this convergence issue in a Wasserstein metric. More precisely, we consider the distance
between finite (possibly signed) measures given by
$$
W_1(\mu,\nu):=\sup_{\|\varphi\|_{\text{Lip}} \leq 1}\biggl( \int \varphi\, d\mu - \int \varphi \,d\nu\biggr),
$$
where $\|\cdot\|_{\text{Lip}}$ stand for the usual Lipschitz semi-norm and which is referred to as the $1$-Wasserstein distance (see for instance \cite{Vil03}).
We recall that $W_1$ induces the weak topology on the space of Borel probability measures
with finite first moment, that we denote by $\mathcal{P}_1(\bt \times \br)$. Notice that, since $\bt$ is compact, this corresponds to measures $\mu$ with finite first moment in velocity.

As it will also be clear from our arguments, $W_1$ is particularly suited to estimate the distance between solutions 
to kinetic equations.
Indeed, for Vlasov-Poisson equations, it is very natural to consider atomic solutions (that it, measures concentrated on finitely many points)
and $W_1$ is able to control the distance between the supports,
while other classical distances (as for instance the total-variation) are too rough
for this (recall that the total-variation distance between two Dirac masses is always $2$ unless they coincide).

Before stating our convergence results, we first deal with the existence of global weak solutions in $\mathcal{P}_1(\bt \times \br)$.

\begin{thm}
\label{thm-exi}
Let $f_0$ be a probability measure in $\mathcal{P}_1(\bt \times \br)$, that is,
\begin{equation}
\label{eq:moment}
\int |v|\,df_0(x,v)<\infty.
\end{equation}
Then there exists a global weak solution to \eqref{vpme-quasi} with initial datum $f_0$.
\end{thm}
The analogous result for the classical Vlasov-Poisson equation was proved by Zheng and Majda \cite{ZM}, and more recently by Hauray \cite{HauX} with a new proof.

We shall prove Theorem~\ref{thm-exi} by combining the method introduced by Hauray (see \cite{HauX})
with new stability estimates for the massless electron system.

\bigskip

Roughly speaking,
the main results of this paper are the following: if we consider initial data for \eqref{vpme-quasi} of the form 
$f_{0,\e} = g_{0,\e} + h_{0,\e}$ with $g_{0,\e}$ analytic (or equal to a finite sum of Dirac masses in velocity, with analytic moments)
and $h_{0,\e}$ converging very fast to $0$ in the $W_1$ distance, then the solution starting from $f_{0,\e}$
converges to the solution of \eqref{formal} with initial condition $g_0 := \lim_{\e \to 0} g_{0,\e}$.
This means that small perturbations in the $W_1$ distance do not affect the quasineutral limit.
Notice that the fact that the size of the perturbation has to be small only in $W_1$
means that $h_{0,\e}$ could be arbitrarily large in any $L^p$ norm.

To state our main results, we first introduce some notation.
The following analytic norm has been used by Grenier \cite{Gr96}
to show convergence results for the quasineutral limit
in the context of the classical Vlasov-Poisson system.

Such a norm is useful to study the quasineutral limit since the formal limit is false in general in Sobolev regularity (see Proposition \ref{thm3} and the discussion below); one can also see that the formal limit equation exhibits a loss of derivative (in the force term), which can be overcome with analytic regularity. 

\begin{defi}
\label{def:norm}
Given $\delta>0$ and a function $g:\bt \to \mathbb R$, we define
$$
\| g \|_{B_\delta} := \sum_{k \in \bz} | \widehat{g}(k) | \delta^{|k|},
$$
where $\widehat{g}(k)$ is the $k$-th Fourier coefficient of $g$. We also define $B_\delta$ as the space of functions $g$ such that $\| g \|_{B_\delta}<+\infty$.
\end{defi}

\begin{thm}
\label{thm1}
Consider a sequence of non-negative initial data in $\mathcal{P}_1(\bt \times \br)$ for \eqref{vpme-quasi} of the form
$$
f_{0,\e} = g_{0,\e} + h_{0,\e},
$$
where $(g_{0,\e})$ is a sequence of continuous functions
satisfying
$$
\sup_{\e\in (0,1)}\sup_{v \in \br} \, (1+v^2) \| g_{0,\e} (\cdot,v)\|_{B_{\delta_0}}  \leq C,
$$
$$
\sup_{\e\in (0,1)} \left\|  \int_\br g_{0,\e}(\cdot,v) \, dv  -1 \right\|_{B_{\delta_0}} < \eta,
$$
for some $\delta_0,C,\eta>0$, with $\eta$ small enough, and admitting a limit $g_0$ in the sense of distributions.

There exists a function $\varphi: \br^+ \to \br^+$, with $\lim_{\e \to 0^+} \varphi(\e) =0$, such that the following holds.

Assume that $(h_{0,\e})$ is a sequence of measures with finite first moment, satisfying 
$$
\forall \e>0, \quad W_1(h_{0,\e},0) \leq \varphi(\e).
$$


Then there exist $T>0$ and $g(t)$ a weak solution on $[0,T]$ of \eqref{formal} with initial condition $g_0 = \lim_{\e \to 0} g_{0,\e}$, such that, for any global weak solution $f_\e(t)$ of \eqref{vpme-quasi} with initial condition $f_{0,\e}$,
$$
\sup_{t \in [0,T]}  W_1(f_\e(t), g(t)) \to_{\e \to 0} 0.
$$

We can explicitly take $\varphi(\e)= \frac{1}{\e} \exp \left( \frac{\lambda}{\e^3}  \exp \frac{15}{2 \e^2}\right)$
for some $\lambda<0$.

\end{thm}

We now state an analogous result for initial data consisting of a finite sum of Dirac masses in velocity:
\begin{thm}
\label{thm2}
Let $N\geq 1$ and consider a sequence of non-negative initial data in $\mathcal{P}_1(\bt \times \br)$ for \eqref{vpme-quasi} of the form
\begin{align*}
f_{0,\e} &= g_{0,\e} + h_{0,\e}, \\
g_{0,\e}(x,v) &= \sum_{i=1}^N \rho_{0,\e}^i(x) \delta_{v=v_{0,\e}^i(x)},
\end{align*}
where the $(\rho_{0,\e}^i, \, v_{0,\e}^i)$ is a sequence of analytic functions
satisfying
$$
\sup_{\e\in (0,1)}\sup_{i \in \{ 1, \cdots , N\}} \| \rho^i_{0,\e}\|_{B_{\delta_0}} +  \| v^i_{0,\e}\|_{B_{\delta_0}}   \leq C,
$$
$$
\sup_{\e\in (0,1)} \left\|  \sum_{i=1}^N  \rho^i_{0,\e}   -1 \right\|_{B_{\delta_0}} < \eta
$$
for some $\delta_0,C,\eta>0$, with $\eta$ small enough, 
and admitting limits $(\rho_0^i, v_0^i)$ in the sense of distributions.

There exists a function $\varphi: \br^+ \to \br^+$, with $\lim_{\e \to 0^+} \varphi(\e) =0$, such that the following holds.

Assume that $(h_{0,\e})$ is a sequence of measures with finite first moment, satisfying 
$$
\forall \e>0, \quad W_1(h_{0,\e},0) \leq \varphi(\e).
$$


Then there exist $T>0$, such that, for any global weak solution $f_\e(t)$ of \eqref{vpme-quasi} with initial condition $f_{0,\e}$,
$$
\sup_{t \in [0,T]}  W_1(f_\e(t), g(t)) \to_{\e \to 0} 0,
$$
where
$$
g(t,x,v) = \sum_{i=1}^N \rho^i(t,x) \delta_{v=v^i(t,x)},
$$
and $(\rho^i, v^i)$ satisfy the multi fluid isothermal system on $[0,T]$
\begin{equation}
\left\{ \begin{array}{ccc}\pt_t \rho^i+ \pt_x (\rho^i v^i)=0,  \\
\pt_t v^i + v^i \pt_x v^i = E, \\
E=-U', \\
U= \log\left(  \sum_{i=1}^N  \rho^i \right),\\
\rho^i\vert_{t=0}= \rho_{0}^i, v^i\vert_{t=0}= v_{0}^i.
\end{array} \right.
\end{equation}

We can explicitly take $\varphi(\e)= \frac{1}{\e} \exp \left( \frac{\lambda}{\e^3}  \exp \frac{15}{2 \e^2}\right)$
for some $\lambda<0$.

\end{thm}

\begin{remark}
\label{rmk:VP}
It is worth mentioning that the previous convergence results can be slightly improved when 
dealing with the classical Vlasov-Poisson equation.
Indeed, thanks to Remark \ref{rmk:VP ws}, 
the analogue of Theorems \ref{thm1} and \ref{thm2} holds for a larger class of initial data.
In fact, it is possible to take
$$\varphi(\e) = \frac{1}{\e} \exp \left( \frac{\lambda}{\e}\right)$$
for some $\lambda<0$.
\end{remark}

In the following, we shall say that a function $\varphi$ is admissible if it can be chosen in the statements of Theorems \ref{thm1} and \ref{thm2}.

The interest of these results is the following: they prove that it is possible to justify the quasineutral limit without making analytic regularity or stability assumption. The price to pay is that the constants involved in the explicit functions $\varphi$ above are extremely small, so that we are very close to the analytic regime. 
We see the ``double exponential'' as an {upper bound} on all admissible $\varphi$,\footnote{Notice that,  as observed in Remark \ref{rmk:VP}, the upper bound on $\varphi$ for the classical Vlasov-Poisson system is only exponential.} and hope our study will be a first step towards refined bounds. 

On the other hand one should have in mind the following negative result, which roughly means that the functions $\varphi(\e)= \e^s$, for any $s>0$ are not admissible (this therefore yields a {lower bound} on admissible functions); this is the consequence of \emph{instability} mechanisms described in \cite{Gr99} and \cite{HKH}.

\begin{prop}
\label{thm3}
There exist smooth homogeneous equilibria $\mu(v)$ such that the following holds.
For any $N>0$ and $s>0$, there exists a 
sequence of non-negative initial data $(f_{0,\e})$ such that 
$$
\| f_{\e,0}- \mu\|_{W^{s,1}_{x,v}} \leq \e^N,
$$
and denoting by $(f_\e)$ the sequence of solutions
to \eqref{vpme-quasi} with initial data $(f_{0,\e})$, for $\alpha \in [0,1)$, the following holds:
$$
\liminf_{\e \rightarrow 0} \sup_{t \in [0,\e^\alpha]} W_1(f_\e(t), \mu) >0.
$$

\end{prop}

We can make the following observations.

\begin{itemize}
\item In Proposition \ref{thm3}, one can take some equilibrium $\mu$ satisfying the same regularity as in Theorem \ref{thm1}. However, there is no contradiction with our convergence results since in Theorem \ref{thm1}
we assume that $g_{0,\e}$ approximates in an analytic way $g_0$ and that $h_{0,\e}$ converges  faster than any polynomial in $\e$. Therefore, the quantification of the ``fast'' convergence in Theorem  \ref{thm1} is important.


\item Note that we can have $W_1(h_{0,\e},0) = o_{\e \to 0}  \left(\frac{1}{\e} \exp \left(  \frac{\lambda}{\e^3}  \exp \frac{15}{2\e^2}\right) \right)$, but 
$$
\| h_{0,\e} \|_{L^p} \sim 1\qquad \text{for any $p \in [1,\infty]$},
$$
as fast convergence to $0$ in the $W_1$ distance can be achieved for sequences exhibiting fast oscillations.
\end{itemize}

Theorem \ref{thm2} can also be compared to the following result in the \emph{stable} case, that corresponds to initial data consisting of one single Dirac mass (see \cite{HK}). In this case, the analogue of Theorem \ref{thm2} can be proved with weak assumptions on the initial data.

\begin{prop}
\label{thm3-2}
Consider
$$
g_{0}(x,v) = \rho_{0}(x) \delta_{v=u_{0}(x)}.
$$
where $\rho_0 >0$ and $\rho_0, u_0 \in H^s(\bt)$, for $s\geq 2$. Consider a sequence $(f_{0,\e})$ of non-negative initial data  in $L^1\cap L^\infty$  for \eqref{vpme-quasi} such that, for all $\e>0$,
$$
 \frac{1}{2}\int f_{0,\e}\vert v \vert ^2 dv dx 
 + \int \left(e^{U_{0,\e}} \log  e^{U_{0,\e}}   - e^{U_{0,\e}} +1\right)\,dx + \frac{\e^2}{2} \int \vert  U'_{0,\e} \vert^2 dx \leq C
$$
for some $C>0$,
and $U_{0,\e}$ is the solution to the Poisson equation
 $$
 \e^2U_{0,\e}''=e^{U_{0,\e}}- \int f_{0,\e} \, dv.
 $$
Also, assume that
 \begin{multline*}
\frac{1}{2}\int g_{0,\e}\vert v - u_0\vert ^2 dv dx 
+ \int \left(e^{U_{0,\e}} \log \left( e^{U_{0,\e}}/ \rho_0 \right)  - e^{U_{0,\e}} +\rho_0\right)\,dx 
+ \frac{\e^2}{2} \int \vert  U'_{0,\e} \vert^2 dx \to_{\e \to 0} 0,
\end{multline*}
Then there exists $T>0$ such that for any  global weak solution $f_\e(t)$ of \eqref{vpme-quasi} with initial condition $f_{0,\e}$, 
$$
\sup_{t \in [0,T]}  W_1(f_\e(t), g(t)) \to_{\e \to 0} 0,
$$
where   
$$
g(t,x,v) =  \rho(t,x) \delta_{v=u(t,x)},
$$
and $(\rho, u)$ satisfy the isothermal Euler system on $[0,T]$
\begin{equation}
\left\{ \begin{array}{ccc}\pt_t \rho+ \pt_x (\rho u)=0,  \\
\pt_t u + u \pt_x u = E, \\
E=-U', \\
U= \log  \rho,\\
\rho\vert_{t=0}= \rho_{0}, u\vert_{t=0}= u_{0}.
\end{array} \right.
\end{equation}
\end{prop}

\begin{remark}
We could also have stated another similar result using the estimates around \emph{stable symmetric homogeneous equilibria} of \cite{HKH}, but will not do so for the sake of conciseness.
\end{remark}

In what follows, we study the quasineutral limit by using Wasserstein stability estimates for the Vlasov-Poisson system.
Such stability estimates were proved for the classical Vlasov-Poisson system by Loeper \cite{Loe}, in  dimension three.
In the one-dimensional case, they can be improved, as recently shown by Hauray in the note \cite{HauX}.

The key estimate is a weak-strong stability result for the (VPME)$_\e$ system, which basically consists in an adaptation of Hauray's proof, and which we believe is of independent interest.

\begin{thm}
\label{thm4}
Let $T>0$. Let $f_\e^1,f_\e^2$ be two measure solutions of \eqref{vpme-quasi} on $[0,T]$, and assume that
$\rho_\e^1(t,x):=\int f_\e^1(t,x,v)\,dv$ is bounded in $L^\infty$ on $[0,T]\times \mathbb T$. Then, for all $\e \in (0,1]$, for all $t \in [0,T]$,
$$
W_1(f_\e^1(t),f_\e^2(t))
\leq \frac{1}{\e} e^{\frac1\e \left[(1 +3/\e^2 )t +(8+ \frac{1}{\e^2} e^{15/(2\e^2)} \Black) \int_0^t\|\rho_\e^1(\tau)\|_\infty\,d\tau\right]}W_1(f_\e^1(0),f_\e^2(0)).
$$
\end{thm}

The proofs of Theorems \ref{thm1} and \ref{thm2} rely on this stability estimate and on a method introduced by Grenier \cite{Gr96} to justify the quasineutral limit for initial data with uniform analytic regularity.

\bigskip

This paper is organized as follows. In Section \ref{sec:wasserstein}, we start by proving Theorem \ref{thm4}.  We then turn to the global weak existence theory in $\mathcal{P}_1(\bt \times \br)$: in Section \ref{sec:exi}, we prove Theorem~\ref{thm-exi}, using some estimates exhibited in the previous section. Section \ref{sec:proofs} is then dedicated to the proof of the main Theorems \ref{thm1} and \ref{thm2}.
We conclude the paper with the study of auxiliary results: in Section \ref{sec:insta} we prove Proposition \ref{thm3}, while in Section \ref{sec:sta} we prove Proposition \ref{thm3-2}. 


\section{Weak-strong stability for the VP system with massless electrons: proof of Theorem \ref{thm4}}
\label{sec:wasserstein}
In this section we prove Theorem \ref{thm4}, i.e., the weak-strong stability estimate for solutions of the (VPME)$_\e$ system. 
Notice that our weak-strong stability  estimate encloses in particular the case
\eqref{vpme} by taking $\e=1$.

Let us introduce the setup of the problem. 
We follow the same notations as in \cite{HauX}.
In particular, we will use a Lagrangian formulation of the problem.

As a preliminary step, it will be convenient to split the electric field in a singular part behaving as the electric field in Vlasov-Poisson
and a regular term.
More precisely,
let us decompose $E_\e$ as $\bar E_\e+\widehat E_\e$ where
$$
\bar E_\e=-\bar U_\e',\qquad \widehat E_\e=-\widehat U_\e',
$$
and $\bar U_\e$ and $\widehat U_\e$ solve respectively
$$
\e^2\bar U_\e''=1-\rho_\e ,\qquad \e^2\widehat U_\e''=e^{\bar U_\e+\widehat U_\e} - 1.
$$
Notice that in this way $U_\e:=\bar U_\e+\widehat U_\e$ solves
$$
\e^2U_\e''=e^{U_\e}-\rho_\e.
$$
Then we can rewrite \eqref{vpme-quasi} as
$$
(VPME)_\e:= \left\{ \begin{array}{ccc}\pt_t f_\e+v\cdot \pt_x f_\e+ (\bar E_\e+\widehat E_\e)\cdot \pt_v f_\e=0,  \\
\bar E_\e=-\bar U_\e',\qquad \widehat E_\e=-\widehat U_\e', \\
\e^2\bar U_\e''=1-\rho_\e,\\
\e^2\widehat U_\e''=e^{\bar U_\e+\widehat U_\e} - 1,\\
f_\e(x,v,0)\ge0,\ \  \int f_\e(x,v,0)\,dx\,dv=1.
\end{array} \right.
$$

To prove Theorem \ref{thm4},
we shall first show a weak-strong stability estimate for a
rescaled system (see (VPME)$_{\e,2}$ below), and then obtain our result by a further scaling argument.

\subsection{A scaling argument}
Let us define
$$
\f_\e(t,x,v):=\frac{1}{\e} f_\e\biggl(\e t,x,\frac{v}{\e}\biggr).
$$
Then a direct computation gives 
$$
(VPME)_{\e,2}:=  \left\{ \begin{array}{ccc}\pt_t \f_\e+v\cdot \pt_x \f_\e+ (\bar \E_\e+\widehat \E_\e)\cdot \pt_v \f_\e=0,  \\
\bar \E_\e=-\bar \U_\e',\qquad \widehat \E_\e=-\widehat \U_\e', \\
\bar \U_\e''=1-\rrho_\e,\\
\widehat \U_\e''=e^{(\bar \U_\e+\widehat \U_\e)/  \e^2 \Black} - 1,\\
\f_\e(x,v,0)\ge0,\ \  \int \f_\e(x,v,0)\,dx\,dv=1,
\end{array} \right.
$$
where
$$
\rrho_\e(t,x):=\int \f_\e(t,x,v)\,dv.
$$
We remark that $\bar \U_\e$ is just the classical Vlasov-Poisson potential so, as in \cite{HauX},
$$
\bar \E_\e(t,x)=-\int_{\mathbb T} W'(x-y)\rrho_\e(t,y)\,dy,
$$
where 
$$
W(x):=\frac{x^2-|x|}{2}
$$
(recall that we are identifying $\mathbb T$ with $[-1/2,1/2)$ with periodic boundary conditions).
In addition, since $W$ is $1$-Lipschitz and $|W|\leq 1$,
recalling that
\be
\label{eq:U}
\bar \U_\e(t,x)=\int_{\mathbb T} W(x-y)\rrho_\e(t,y)\,dy
\ee
we see that
$\bar \U_\e$ is $1$-Lipschitz and $|\bar \U_\e|\leq 1$.\\

\subsection{Weak-strong estimate for the rescaled system}
The goal of this section is to prove a quantitative weak-strong convergence for the rescaled system
(VPME)$_{\e,2}$.
In order to simplify the notation, we omit the subscript $\e$.
In the sequel we will need the following elementary result:
\begin{lem} \label{lem:media nulla}
Let $h: [-1/2,1/2] \to \br$ be a continuous  function such that $\int_{-1/2}^{1/2} h=0.$ Then 
$$
\|h\|_{\infty}\le \int_{-1/2}^{1/2} |h'|.
$$
\end{lem}

\begin{proof}
Since $\int_{-1/2}^{1/2} h=0$ there exists a point $\bar x$ such that $h(\bar{x})=0.$ Then, by the Fundamental  Theorem of Calculus,

$$
|h(x)|= \Big|\int_{\bar{x}}^x h' \Big|\le \int_{-1/2}^{1/2} |h'| \qquad \forall\, x \in [-1/2,1/2].
$$

\end{proof}
We can now prove existence of solutions to the equation for $\widehat \U$.
\begin{lem}
\label{lem:hatU}
There exists a unique solution of
\be
\label{eq-lem-Poisson}
\widehat \U''=e^{(\bar \U+\widehat \U)/\e^2} - 1\qquad \text{on }\mathbb T
\ee
and this solution satisfies
$$
\|\widehat \U\|_\infty \leq 3,\quad \|\widehat \U'\|_\infty \leq 2,\quad  \|\widehat \U''\|_\infty \leq \frac{3}{\e^2}.
$$
\end{lem}

\begin{proof}
We prove existence of $\widehat U$ by finding a minimizer for
$$
h \mapsto E[h]:= \int_{\mathbb T} \left(\frac12 (h')^2 + \e^2e^{(\bar \U+h)/\e^2} - h \right) \, dx
$$
among all periodic functions $h :[-1/2,1/2]\to \br$.
Indeed, as we shall see later, the Poisson equation we intend to solve is nothing but the Euler-Lagrange equation of the above functional.

Notice that since $E[h]$ is a strictly convex functional, solutions of the Euler-Lagrange equation
are minimizers and the minimizer is unique.
Let us now prove the existence of such a minimizer.

Take $h_k$ a minimizing sequence, that is
$$
E[h_k] \to \inf_{h}E[h]=:\alpha.
$$
Notice that by choosing $h=-\bar\U$ we get (recall that $|\bar\U|, |\bar\U'|\leq 1$, see \eqref{eq:U})
$$
\alpha \leq E[-\bar\U] =\int_{\mathbb T} \left( \frac12 (\bar\U')^2 +\bar\U \right) \, dx \leq 2,
$$
hence
$$
E[h_k] \leq 3 \qquad \text{for $k$ large enough.}
$$
We first want to prove that $h_k$ is uniformly bounded in $H^1$.

We observe that, since $\bar \U\geq -1$,
for any $s \in \br$
$$
\e^2e^{(\bar \U(x)+s)/\e^2} -s \geq \e^2e^{(s-1)/\e^2}-s.
$$
Now, for $s \geq 2$  (and $\e \in (0,1]$) \Black we have
$$
\e^2e^{(s-1)/\e^2}-s \geq e^{s-1}-s \geq s - 2\log 2 \geq s -3,
$$
while for $s\leq 2$ we have
$$
\e^2e^{(s-1)/\e^2}-s \geq -s \geq |s|-4,
$$
thus 
$$
e^{(s-1)/\e^2}-s \geq |s|- 4 \qquad \forall\,s \in \br.
$$
Therefore
\be
\label{eq:energy}
3 \geq E[h_k] \geq \int_{\mathbb T} \frac12 (h_k')^2 +|h_k| - 4,
\ee
which gives
$$
 \int_{\mathbb T} \frac12 (h_k')^2 \leq 8.
$$
In particular, by the Cauchy-Schwarz inequality this implies
\be
\label{eq:holder}
\begin{split}
|h_k(x)-h_k(z)|& \leq \biggl|\int_z^x|h_k'(y)|\,dy \biggr|\leq \sqrt{|x-z|} \sqrt{\int_{\mathbb T}|h_k'(y)|^2\,dy}\\
&\leq 4\,\sqrt{|x-z|}.
\end{split}
\ee
Up to now we have proved that $h_k'$ are uniformly bounded in $L^2$.
We now want to control $h_k$ in $L^\infty$.
 
Let $M_k$ denote the maximum of $|h_k|$ over $\bt$. Then by \eqref{eq:holder} we deduce that
$$
h_k(x) \geq M_k - 4 \qquad \forall\,x \in \bt,
$$
hence, recalling \eqref{eq:energy},
\begin{align*}
3 &\geq E[h_k] \geq \int_{\mathbb T}( |h_k(x)| - 4) \,dx
\geq  M_k -8,
\end{align*}
which implies $M_k \leq 11$.
Thus, we proved that $|h_k|\leq 11$ for all $k$ large enough,
which implies in particular that $h_k$ are uniformly bounded in $L^{2}$.
 
 In conclusion, we have proved that $h_k$ are uniformly bounded in $H^1$ (both $h_k$ and $h_k'$
 are uniformly bounded in $L^2$) and in addition they are uniformly bounded and uniformly continuous (as a consequence of \eqref{eq:holder}).
 Hence, up to a subsequence, they converge weakly in $H^1$ (by weak compactness of balls in $H^1$) and uniformly (by the Ascoli-Arzel\`a theorem) to a function $\widehat \U$.
 We claim that $\widehat \U$ is a minimizer.
 Indeed, by the lower semicontinuity of the $L^2$ norm under weak convergence,
 $$
\int_{\mathbb T}|\widehat \U'(x)|^2\,dx  \leq \liminf_k \int_{\mathbb T}|h_k'(x)|^2\,dx.
 $$
On the other hand, by uniform convergence,
 $$
\int_{\mathbb T}\left(\e^2 e^{(\bar \U(x)+h_k(x))/\e^2} - h_k(x)\right)\,dx \to 
\int_{\mathbb T}\left( \e^2 e^{(\bar \U(x)+\widehat \U(x))/\e^2} - \widehat \U(x) \right)\,dx.
 $$
 In conclusion
 $$
 E[\widehat \U] \leq \liminf_k E[h_k]=\alpha,
 $$
 which proves that $\widehat \U$ is a minimizer.
 
 By the minimality,
 $$
 0=\frac{d}{d\eta}\bigg|_{\eta=0}E[\widehat\U+\eta h]=\int_{\mathbb T} \left( \widehat \U'\,h' + e^{(\bar \U+\widehat \U)/\e^2}h-h \right) \, dx
 =\int_{\mathbb T}[-\widehat \U''+e^{(\bar \U+\widehat \U)/\e^2} - 1]h \, dx,
 $$
 which proves that $\widehat \U$ solves \eqref{eq-lem-Poisson} by the arbitrariness of $h$.
 
 \bigskip
 
We now prove the desired estimates on $\widehat \U$.
Since $\widehat \U'$ is a periodic continuous function we have
$$
0=\int_{\mathbb T} \widehat \U'' \, dx =\int_{\mathbb T} \(e^{(\bar \U+\widehat \U)/\e^2}-1 \) \, dx.
$$
Thus we get
$$
\int_{\mathbb T} \big| \widehat \U''\big| \, dx \leq \int_{\mathbb T}\Big| \(e^{(\bar \U+\widehat \U)/\e^2}-1 \)  \Big| \, dx\leq \int_{\mathbb T} e^{(\bar \U+\widehat \U)/\e^2}\, dx +1=2,
$$
and so, by Lemma \ref{lem:media nulla}, we deduce 
$$
\|\widehat \U'\|_\infty \le  \int_{\mathbb T} |\widehat \U''| \, dx \le 2.
$$
Since $\|\widehat \U'\|_\infty \le 2,$ $\|\bar \U\|_\infty\le 1$, and $\int_{\mathbb T} e^{(\bar \U+\widehat \U)/\e^2} \, dx =1$ we claim that $\|\widehat \U \|_\infty \le 3.$
Indeed, suppose that there exists $\bar x$ such that $\widehat \U (\bar x) \ge M.$ Then, recalling that $\|\widehat \U'\|_\infty \le 2,$ we have $\widehat \U (x) \ge M-2$ for all $x$.
Using that $\|\bar \U\|_\infty\le 1$ we get $\widehat \U(x)+\bar \U(x) \ge M-3.$
Then,
$$
1= \int_{\mathbb T} e^{(\bar \U+\widehat \U)/\e^2} \, dx  \ge \int_{\mathbb T} e^{(M-3)/\e^2} \, dx = e^{(M-3)/\e^2} \Rightarrow  M \le 3.
$$
On the other hand, if there exists $\bar x$ such that $\widehat \U (\bar x) \le -M,$ 
then an analogous argument gives
$$
1= \int_{\mathbb T} e^{(\bar \U+\widehat \U)/\e^2} \, dx  \le \int_{\mathbb T} e^{{-(M-3)/\e^2}} \, dx = e^{{-(M-3)/\e^2}} \Rightarrow  M \le 3.
$$
Hence we have that $\|\widehat \U \|_\infty \le 3.$
Finally, to estimate $\widehat \U''$ we 
differentiate the equation 
$$
\widehat \U''= \(e^{(\bar \U+\widehat \U)/\e^2}-1\),
$$
recall that $\|\bar \U'\|_\infty\le 1,$ $\|\widehat \U'\|_\infty\le 2$, and $\int_{\mathbb T}e^{(\bar \U+\widehat \U)/\e^2}=1,$ to obtain
$$
 \int_{\mathbb T}|\widehat \U'''| = \int_{\mathbb T} \Big|e^{(\bar \U+\widehat \U)/\e^2}\biggl(\frac{\bar \U'+\widehat \U'}{\e^2} \biggr)\Big| \le \frac{\|\bar \U'\|_\infty+\|\widehat \U'\|_\infty}{\e^2}\int_{\mathbb T}e^{(\bar \U+\widehat \U)/\e^2} \le \frac{3}{\e^2},
$$
so, by Lemma \ref{lem:media nulla} again, we get
$$
\|\widehat \U''\|_\infty\le  \int_{\mathbb T}|\widehat \U'''|\le  \frac{3}{\e^2},
$$
as desired.
\end{proof}

To prove the weak-strong stability result for (VPME)$_{\e,2}$, following the strategy used in \cite{HauX} 
for the classical Vlasol-Poisson system, we will represent
solutions in Lagrangian variables instead of using the Eulerian formulation.
In this setting, the phase space is $\bb{T}\times\br$ and particles in the phase-space are represented by $Z=(X,V)$, where the random variables $X:[0,1]\to \bb{T}$ and $V:[0,1]\to \br$ are maps from the probability space
$([0,1],ds)$ to the physical space.
The idea is that elements in $[0,1]$ do not have any physical meaning
but they just label the particles $\{(X(s),V(s))\}_{s \in [0,1]}\subset \bb{T}\times\br$.

We mention that this ``probabilistic'' point of view was already introduced for ODEs by Ambrosio in his study of linear transport equations \cite{amb} and generalized later by Figalli to the case of SDEs \cite{fig}.

To any random variable as above, one associates the mass distribution of particles in the phase space as follows:\footnote{Note that the law of $(X,V)$ may not be absolutely continuous, we just wrote the formula to explain
the heuristic.
}
$$
\f(x,v)\,dx\,dv= (X,V)_\# ds,
$$
that is $\f$ is the law of $(X,V)$.
So, instead of looking for the evolution of $\f$, we rather let $Z_t:=(X_t,V_t)$ evolve accordingly to the following Lagrangian system (recall that, to simplify the notation, we are omitting the subscript $\e$)
\begin{equation}
\label{eq:VPLEL}
(VPME)_{L,2}:= \left\{ \begin{array}{cc}
\dot X_t=V_t,\\
\dot V_t=\bar \E(X_t) + \widehat \E(X_t),\\
\bar \E=-\bar \U',\qquad \widehat \E=-\widehat \U', \\
\bar \U''=1-\rho,\\
\widehat \U''=e^{(\bar \U+\widehat \U)/\e^2} - 1,\\
\rho(t)=(X_t)_\# ds.
\end{array} \right.
\end{equation}
(Such a formulation is rather intuitive if one thinks of the evolution of finitely many particles.)
Notice that the fact that $\rho(t)$ is the law of $X_t$ is a consequence of the fact that $\f$ is the law of $Z_t$.

As initial condition we impose that at time zero $Z_t$ is distributed accordingly to $\f_0$, that is
$$
(Z_0)_\#ds=\f_0(x,v)\,dx\,dv.
$$
We recall the following characterization of the $1$-Wasserstein distance, used also by Hauray in \cite{HauX}:
$$
W_1(\mu,\nu)=\min_{X_\#ds=\mu,\,Y_\#ds=\nu} \int_0^1|X(s)-Y(s)|\,ds.
$$
Hence, if $\f_1,\f_2$ are two solutions of (VPME)$_{\e,2}$, in order to control $W_1(\f_1(t),\f_2(t))$, it is sufficient to do the following:
choose $Z_0^1$ and $Z_0^2$ such that
$$
(Z_0^i)_\#ds=d f^i(0,x,v),\qquad i=1,2
$$
and 
$$
W_1(f^1(0),f^2(0))= \int_0^1|Z_0^1(s)-Z_0^2(s)|\,ds,
$$
and prove a bound on $ \int_0^1|Z_t^1(s)-Z_t^2(s)|\,ds$ for $t \geq 0$.
In this way we automatically get a control on
$$
W_1(\f^1(t),\f^2(t))\leq  \int_0^1|Z_t^1(s)-Z_t^2(s)|\,ds.
$$
So, our goal is to estimate $ \int_0^1|Z_t^1(s)-Z_t^2(s)|\,ds$.
For this, as in \cite{HauX} we consider 
$$
\frac{d}{dt} \int_0^1|Z_t^1(s)-Z_t^2(s)|\,ds.
$$
Using \eqref{eq:VPLEL}, this is bounded by
\begin{align*}
&\int_0^1|V_t^1(s)-V_t^2(s)|\,ds+ \int_0^1|\E_t^1(X_t^1)-\E_t^2(X_t^2)|\,ds\\
& \leq \int_0^1|Z_t^1(s)-Z_t^2(s)|\,ds+ \int_0^1|\bar \E_t^1(X_t^1)-\bar \E_t^2(X_t^2)|\,ds
+ \int_0^1|\widehat \E_t^1(X_t^1)-\widehat \E_t^2(X_t^2)|\,ds\\
& \leq \int_0^1|Z_t^1(s)-Z_t^2(s)|\,ds+ 8\|\rrho^1(t)\|_\infty \int_0^1|Z_t^1(s)-Z_t^2(s)|\,ds\\
&+\int_0^1|\widehat \E_t^1(X_t^1)-\widehat \E_t^2(X_t^1)|\,ds+\int_0^1|\widehat \E_t^2(X_t^1)-\widehat \E_t^2(X_t^2)|\,ds,
\end{align*}
where we split $\E_t^1$ and $\E_t^2$ as a sum of $\bar \E_t^1+\widehat \E_t^1$ and
 $\bar \E_t^2+\widehat \E_t^2$, and
 we applied the estimate in \cite[Proof of Theorem 1.8]{HauX}
 to control 
 $$
 \int_0^1|\bar \E_t^1(X_t^1)-\bar \E_t^2(X_t^2)|\,ds
 $$
 by
 $$
 8\|\rrho^1(t)\|_\infty \int_0^1|Z_t^1(s)-Z_t^2(s)|\,ds.
 $$
To estimate the last two terms, we argue as follows:
for the second one we recall that $\widehat \E_t^2$ is $(3/\e^2)$-Lipschitz (see Lemma \ref{lem:hatU}), hence
$$
\int_0^1|\widehat \E_t^2(X_t^1)-\widehat \E_t^2(X_t^2)|\,ds
\leq \frac{3}{\e^2}
\int_0^1|X_t^1-X_t^2|\,ds \leq  \frac{3}{\e^2}\int_0^1|Z_t^1-Z_t^2|\,ds.
$$
For the first term, we 
first observe the following fact: recalling \eqref{eq:U}
and that $W$ is 1-Lipschitz, we have
\be
\label{eq:bound U}
\begin{split}
|\bar \U_t^1-\bar \U_t^2|(x)&=
\biggl|\int_0^1 W(x-X_t^1) - W(x-X_t^2)\,ds \biggr|\\
& \leq \int_0^1 |X_t^1-X_t^2|\,ds\leq \int_0^1|Z_t^1-Z_t^2|\,ds
\end{split}
\ee
for all $x$.
Now we want to estimate $\widehat \E_t^1-\widehat \E_t^2$ in $L^2$:
for this we start from the equation
$$
(\widehat \U_t^1)'' - (\widehat \U_t^2)''=e^{(\bar \U_t^1 +\widehat \U_t^1)/\e^2} - e^{(\bar \U_t^2 +\widehat \U_t^2)/\e^2}.
$$
Multiplying by $\widehat \U_t^1 - \widehat \U_t^2$ and integrating by parts, we get
\begin{align*}
0&=\int_{\mathbb T} \Bigl((\widehat \U_t^1)' - (\widehat \U_t^2)'\Bigr)^2\,dx
+\int_{\mathbb T} \left[e^{(\bar \U_t^1 +\widehat \U_t^1)/\e^2} - e^{(\bar \U_t^2 +\widehat \U_t^2)/\e^2}\right]
[\widehat \U_t^1 - \widehat \U_t^2]\,dx\\
&=\int_{\mathbb T} \Bigl((\widehat \U_t^1)' - (\widehat \U_t^2)'\Bigr)^2\,dx
+\int_{\mathbb T} \left[e^{(\bar \U_t^1 +\widehat \U_t^1)/\e^2} - e^{(\bar \U_t^1 +\widehat \U_t^2)/\e^2}\right]
[\widehat \U_t^1 - \widehat \U_t^2]\,dx\\
&\qquad + \int_{\mathbb T} \left[e^{(\bar \U_t^1 +\widehat \U_t^2)/\e^2} - e^{(\bar \U_t^2 +\widehat \U_t^2)/\e^2}\right]
[\widehat \U_t^1 - \widehat \U_t^2]\,dx.
\end{align*}
For the second term we observe that, by the Fundamental Theorem of Calculus,
$$
e^{(\bar \U_t^1 +\widehat \U_t^1)/\e^2} - e^{(\bar \U_t^1 +\widehat \U_t^2)/\e^2}
= \frac{1}{\e^2} \Black \biggl(\int_0^1 e^{[\bar \U_t^1 +\lambda \widehat \U_t^1+(1-\lambda)\widehat \U_t^2]/\e^2}\,d\lambda\biggr)\,[\widehat \U_t^1 - \widehat \U_t^2].
$$
Hence,
\begin{align*}
&\int_{\mathbb T} \left[e^{(\bar \U_t^1 +\widehat \U_t^1)/\e^2} - e^{(\bar \U_t^1 +\widehat \U_t^2)/\e^2}\right]
[\widehat \U_t^1 - \widehat \U_t^2]\,dx\\
&
=\int_{\mathbb T}  \frac{1}{\e^2} \Black   \biggl(\int_0^1 e^{[\bar \U_t^1 +\lambda \widehat \U_t^1+(1-\lambda)\widehat \U_t^2]/\e^2}\,d\lambda\biggr)\,(\widehat \U_t^1 - \widehat \U_t^2)^2dx\\
&\geq   \frac{1}{\e^2}   e^{-5/\e^2} \Black \int_{\mathbb T}(\widehat \U_t^1 - \widehat \U_t^2)^2dx
\end{align*}
where we used that  $\bar \U$ and $\widehat \U$ are bounded by $1$ and $4$, respectively.

For the third term, we simply estimate 
$$
|e^{(\bar \U_t^1 +\widehat \U_t^2)/\e^2} - e^{(\bar \U_t^2 +\widehat \U_t^2)/\e^2}| \leq  \frac{1}{\e^2} e^{5/\e^2} \Black |\bar \U_t^1 -\bar \U_t^2|,
$$
hence, combining all together,
\begin{align*}
&\int_{\mathbb T} \Bigl((\widehat \U_t^1)' - (\widehat \U_t^2)'\Bigr)^2\,dx
+   e^{-5/\e^2} \Black \int_{\mathbb T}(\widehat \U_t^1 - \widehat \U_t^2)^2dx\\
&\leq  \frac{1}{\e^2} e^{5/\e^2} \Black \int_{\mathbb T}  |\bar \U_t^1 -\bar \U_t^2|\,|\widehat \U_t^1 - \widehat \U_t^2|\,dx
\\
&\leq  \frac{1}{\e^2} e^{5/\e^2} \Black\delta  \int_{\mathbb T}  |\widehat \U_t^1 - \widehat \U_t^2|^2\,dx
+  \frac{1}{\e^2}  \frac{e^{5/\e^2}}{\delta} \Black \int_{\mathbb T}  |\bar \U_t^1 -\bar \U_t^2|^2\,dx.
\end{align*}
Thus, choosing  $\delta:=\e^2  e^{-10/\e^2}$ \Black 
, we finally obtain
$$
\int_{\mathbb T} \Bigl((\widehat \U_t^1)' - (\widehat \U_t^2)'\Bigr)^2\,dx \leq   \frac{1}{\e^4} e^{15/\e^2} \Black  \int_{\mathbb T}  |\bar \U_t^1 -\bar \U_t^2|^2\,dx
$$
Observing now that $(\widehat \U_t^i)'=-\widehat \E_t^1$ and recalling \eqref{eq:bound U},
we obtain
\begin{equation}
\label{eq:continuity E}
\sqrt{\int_{\mathbb T} \Bigl(\widehat \E_t^1 - \widehat \E_t^2\Bigr)^2\,dx} \leq   \frac{1}{\e^2} e^{15/(2\e^2)} \Black \int_0^1|Z_t^1-Z_t^2|\,ds.
\end{equation}

Thanks to this estimate, we conclude that 
\begin{align*}
\int_0^1|\widehat \E_t^1(X_t^1)-\widehat \E_t^2(X_t^1)|\,ds&
=\int_{\mathbb T}|\widehat \E_t^1(x)-\widehat \E_t^2(x)|\,\rrho^1(t,x)\,dx\\
&\leq \|\rrho^1(t)\|_\infty\int_{\mathbb T}|\widehat \E_t^1-\widehat \E_t^2|\,dx\\
&\leq \|\rrho^1(t)\|_\infty\sqrt{\int_{\mathbb T}|\widehat \E_t^1-\widehat \E_t^2|^2\,dx}\\
&\leq   \frac{1}{\e^2} e^{15/(2\e^2)} \Black \|\rrho^1(t)\|_\infty \int_0^1|Z_t^1-Z_t^2|\,ds.
\end{align*}
Hence, combining all together we proved that
$$
\frac{d}{dt} \int_0^1|Z_t^1-Z_t^2|\,ds
\leq \Bigl(1+8\|\rrho^1(t)\|_\infty+\frac{3}{\e^2} + \frac{1}{\e^2} e^{15/(2\e^2)} \Black\|\rrho^1(t)\|_\infty \Bigr) \int_0^1|Z_t^1-Z_t^2|\,ds,
$$
so that, by Gronwall inequality,
$$
\int_0^1|Z_t^1-Z_t^2|\,ds
\leq e^{(1 +3/\e^2 )t +(8+ \frac{1}{\e^2} e^{15/(2\e^2)} \Black) \int_0^t\|\rrho^1(\tau)\|_\infty\,d\tau}\int_0^1|Z_0^1-Z_0^2|\,ds,
$$
which implies (recalling the discussion at the beginning of this computation)
\be
\label{eq:ws rescaled}
W_1(\f^1(t),\f^2(t))
\leq e^{(1 +3/\e^2 )t +(8+ \frac{1}{\e^2} e^{15/(2\e^2)} \Black) \int_0^t\|\rrho^1(\tau)\|_\infty\,d\tau}W_1(\f^1(0),\f^2(0)).
\ee
This proves the desired weak-strong stability for the rescaled system.

\subsection{Back to the original system and conclusion of the proof}
To obtain the weak-strong stability estimate for our original system,
we simply use \eqref{eq:ws rescaled} together with the definition of $W_1$.
More precisely, given two densities $f_1(x,v)$ and $f_2(x,v)$,
consider
$$
\f_i(x,v):=\frac1\e f_i(x,v/\e),\qquad i=1,2.
$$
Then
\begin{align*}
W_1(\f_1,\f_2)&=\sup_{\|\varphi\|_{\text{Lip}} \leq 1} \int \varphi(x,v) [\f_1(x,v) - \f_2(x,v)]\,dx\,dv\\
&=\sup_{\|\varphi\|_{\text{Lip}} \leq 1} \int \varphi(x,v) \frac{1}{\e}[f_1(x,v/\e) - f_2(x,v/\e)]\,dx\,dv\\
&=\sup_{\|\varphi\|_{\text{Lip}} \leq 1} \int \varphi(x,\e w) [f_1(x,w) - f_2(x,w)]\,dx\,dw.
\end{align*}
We now observe that if $\varphi$ is $1$-Lipschitz so is $\varphi(x,\e w)$ for $\e \leq 1$, hence
\begin{multline*}
\sup_{\|\varphi\|_{\text{Lip}} \leq 1} \int \varphi(x,\e w) [f_1(x,w) - f_2(x,w)]\,dx\,dw\\
\leq \sup_{\|\psi\|_{\text{Lip}} \leq 1} \int \psi(x,w) [f_1(x,w) - f_2(x,w)]\,dx\,dw=W_1(f_1,f_2).
\end{multline*}
Reciprocally, given any 1-Lipschitz function $\psi(x,w)$, the function $\varphi(x,w):=\e \psi(x,w/\e)$
is still 1-Lipschitz, hence 
\begin{multline*}
W_1(f_1,f_2)= \sup_{\|\psi\|_{\text{Lip}} \leq 1} \int \psi(x,w) [f_1(x,w) - f_2(x,w)]\,dx\,dw\\
\leq \sup_{\|\varphi\|_{\text{Lip}} \leq 1} \int \frac{1}{\e}\varphi(x,\e w) [f_1(x,w) - f_2(x,w)]\,dx\,dw = 
\frac{1}{\e} W_1(\f_1,\f_2).
\end{multline*}
Hence, in conclusion, we have
$$
\e W_1(f_1,f_2) \leq W_1(\f_1,\f_2) \leq W_1(f_1,f_2).
$$
In particular, when applied to solutions of  (VPME), we deduce that
\be
\label{eq:eps W1}
\e W_1(f_1(\e t),f_2(\e t)) \leq W_1(\f_1(t),\f_2(t)) \leq W_1(f_1(\e t),f_2(\e t)).
\ee
Observing that
$$
\int_0^t \|\rrho^1(\tau)\|_\infty\,d\tau = \int_0^t \|\rho^1(\e \tau)\|_\infty\,d\tau
= \frac{1}{\e} \int_0^{\e t} \|\rho^1(\tau)\|_\infty\,d\tau,
$$
\eqref{eq:eps W1} together with \eqref{eq:ws rescaled} gives
$$
W_1(f^1(t),f^2(t))
\leq \frac{1}{\e} e^{\frac1\e \left[(1+3/\e^2 )t +(8+ \frac{1}{\e^2} e^{15/(2\e^2)} \Black) \int_0^t\|\rho^1(\tau)\|_\infty\,d\tau\right]}W_1(f^1(0),f^2(0)),
$$
which concludes the proof of Theorem \ref{thm4}.

\begin{remark}
\label{rmk:VP ws}
Notice that, if we were working with the classical Vlasov-Poisson system, the stability estimate would have simply been
$$
W_1(\f^1(t),\f^2(t))
\leq e^{t +8 \int_0^t\|\rrho^1(\tau)\|_\infty\,d\tau}W_1(\f^1(0),\f^2(0)),
$$
(compare with \cite{HauX}), so in terms of $f$
$$
W_1(f^1(t),f^2(t))
\leq \frac{1}{\e} e^{\frac1\e \left[t +8 \int_0^t\|\rho^1(\tau)\|_\infty\,d\tau\right]}W_1(f^1(0),f^2(0)),
$$
\end{remark}

\section{Proof of Theorem~\ref{thm-exi}}
\label{sec:exi}

In this section, we prove the existence of global weak solutions in $\mathcal{P}_1(\bt \times \br)$ for the (VPME) system.
Without loss of generality we prove the existence of solutions when $\e=1$ (that is,
for \eqref{vpme}).

To prove existence of weak solutions we follow \cite[Proposition 1.2 and Theorem 1.7]{HauX}.
For this, we take a random variable $(X_0,V_0):[0,1]\to \mathbb T \times \br $ whose law is $f_0$, that is
$(X_0,V_0)_\# ds=f_0$, and we solve
\begin{equation}
\label{eq:VPLELappendix}
(VPME)_{L,2}:= \left\{ \begin{array}{cc}
\dot X_t=V_t,\\
\dot V_t=\bar \E(X_t) + \widehat \E(X_t),\\
\bar \E=-\bar \U',\qquad \widehat \E=-\widehat \U', \\
\bar \U''=1-\rho,\\
\widehat \U''=e^{(\bar \U+\widehat \U)} - 1,\\
\rho(t)=(X_t)_\# ds.
\end{array} \right.
\end{equation}
Indeed, once we have $(X_t,V_t)$, $f_t:=(X_t,V_t)_\# ds$ will be a solution to \eqref{vpme}.
We split the argument in several steps.\\

{\it Step 1: Solution of the $N$ particle system (compare with \cite[Proof of Proposition 1.2]{HauX}).}
We start from the $N$ particle systems of ODEs, $i=1,\ldots,N$,
$$
\left\{ \begin{array}{cc}
\dot X^i_t=V_t^i,\\
\dot V_t^i=\bar E(X_t^i) + \widehat E(X_t^i),\\
\bar E=-\bar U',\qquad \widehat E=-\widehat U', \\
\bar U''=\frac{1}{N}\sum_{i=1}^N \delta_{X^i} -1,\\
\widehat U''=e^{\bar U+\widehat U} - 1,\\
\end{array} \right.
$$
Because the electric field $\bar E(X^i)=-\frac{1}{N}\sum_{j\neq i} W'(X^i-X^j)$ is singular when $X^i=X^j$ for some $i \neq j$,
to prove existence we want to rewrite the above system as a differential inclusion
$$
\dot {\mathcal Z}^N(t) \in \mathcal B^N(\mathcal Z^N(t)),
$$
where $\mathcal B^N$ is a multivalued map from $\br^{2N}$ into the set of parts of $\br^{2N}$.
For this, we write
$$
\dot {\mathcal Z}^N(t) \in \mathcal B^N(\mathcal Z^N(t))
\quad \Leftrightarrow \quad \dot X^i=V^i,
\quad \dot V^i \in \frac{1}{N}\sum_{j=1}^N F_{ij}
$$
where 
$$
F_{ij}=-F_{ji}=-W'(X^i-X^j)+\widehat E(X^i) \quad \text{when $X^i \neq X^j$},
$$
$$
F_{ij}=-F_{ji} \in [\widehat E(X^i)-1/2, \widehat E(X^i)+1/2] \quad \text{when $X^i = X^j$}.
$$
As in \cite{HauX}, this equation is solved by Filippov's Theorem  (see \cite{Fil}) which provides existence of a solution,
and as shown in \cite[Step 2 of proof of Proposition 1.2]{HauX} the solution of the differential inclusion is a solution to the $N$ particle problem.\\

{\it Step 2: Approximation argument.}
To solve $(VPME)_{L,2}$ we approximate $f_0$ with a family of empirical measures
$$
f_0^N:=\frac{1}{N}\sum_{i=1}^N \delta_{(x^i,v^i)},
$$
that we can assume to satisfy (thanks to \eqref{eq:moment})
\begin{equation}
\label{eq:moment N}
\int |v|\,df_0^N(x,v)\leq C\qquad \forall\,N,
\end{equation}
and we apply Step 1 to solve the ODE system and find solutions $(X_t^N,V_t^N) \in \mathbb T\times \br$
of $(VPME)_{L,2}$ starting from an initial condition $(X^N_0,V^N_0)$ whose law is $f_0^N$.

Next, we notice that in \cite[Step 2, Proof of Theorem 1.7]{HauX} the only property on the vector field used in the proof is the fact that $F_{ij}$ are bounded by $1/2$, and it is used to show that
$$
\sup_{u,s \in [0,t]} \frac{|Z^N(u)-Z^N(u)|}{|s-u|} \leq |V_0^N|+ \frac{1}{2}(1+t),
$$
which, combined with \eqref{eq:moment N}, is enough to ensure tightness (see \cite[Step 2, Proof of Theorem 1.7]{HauX} for more details).
Since in our case the $F_{ij}$ are also bounded (as we are simply adding a bounded term $\widehat E$),
we deduce that for some $C>0$,
$$
\sup_{u,s \in [0,t]} \frac{|Z^N(u)-Z^N(u)|}{|s-u|} \leq |V_0^N|+ C(1+t),
$$
so the sequence $Z^N:=(X^N,V^N)$
is still tight and (up to a subsequence) converge to a process $Z=(X,V)$: it holds
\begin{equation}
\label{eq:tightness}
\int_0^1 \sup_{t \in [0,T]} \bigl|Z^N_t(s) - Z_t(s)\bigr| \,ds \to 0 \qquad \text{as $N \to \infty$}
\end{equation}
for any $T>0$.\\

{\it Step 3: Characterization of the limit process}.
We now want to prove that the limit process $Z_t=(X_t,V_t)$ is a solution of $(VPME)_{L,2}$.

Let us denote by $\bar E^N$ and $\widehat E^N$ the electric fields associated to the solution $(X^N,V^N)$.
Recall that $(X^N,V^N)$ solve
$$
\dot X^N=V^N,\qquad \dot V^N=\bar E^N(X^N)+\widehat E^N(X^N),
$$
or equivalently
$$
X^N_t=\int_0^t V^N_\tau\,d\tau,\qquad V^N_t=\int_0^t\bar E^N_\tau(X^N_\tau)+\widehat E^N_\tau(X^N_\tau)\,d\tau.
$$
In \cite[Step 3, Proof of Theorem 1.7]{HauX}, using \eqref{eq:tightness}, it is proved that 
$$\int_0^t \bar E^N_\tau(X^N_\tau)\rightarrow \int_0^t \bar E_\tau(X_\tau) \quad \text{ in  } L^1([0,1], \,ds),$$ 
so, to ensure that $(X,V)$ solves 
$$
X_t=\int_0^t V_\tau\,d\tau,\qquad V_t=\int_0^t\bar E_\tau(X_\tau)+\widehat E_\tau(X_\tau)\,d\tau,
$$
it suffices to show that, for any $\tau \geq 0$,
$$
\int_0^1 |\widehat E^N_\tau(X^N_\tau(s))-\widehat E_\tau(X_\tau(s))|\,ds \to 0 \qquad \text{as $N \to \infty$.}
$$
To show this we see that
\begin{align*}
\int_0^1 |\widehat E^N_\tau(X^N_\tau(s))-\widehat E_\tau(X_\tau(s))|\,ds& \leq 
\int_0^1 |\widehat E^N_\tau(X^N_\tau(s))-\widehat E_\tau(X_\tau^N(s))|\,ds\\
&+\int_0^1 |\widehat E_\tau(X^N_\tau(s))-\widehat E_\tau(X_\tau(s))|\,ds\\
&=:I_1+I_2
\end{align*}
For $I_2$ we use that $\widehat E_\tau$ is $M$-Lipschitz (recall Lemma \ref{lem:hatU})
to estimate
$$
I_2\leq M \int_0^1 |X^N_\tau(s)-X_\tau(s)|\,ds
$$
that goes to $0$ thanks to \eqref{eq:tightness}.

For $I_1$, we notice that the Cauchy-Schwarz inequality, \eqref{eq:continuity E}, and \eqref{eq:tightness} imply that, as $N \to \infty$,
\begin{align*}
\int_{\mathbb T} |\widehat E^N_\tau(x)-\widehat E_\tau(x)|\,dx& \leq
\sqrt{\int_{\mathbb T} |\widehat E^N_\tau(x)-\widehat E_\tau(x)|^2\,dx}\\
&\leq  
\bar C\int_0^1|Z_\tau^N(s)-Z_\tau(s)|\,ds \to 0.
\end{align*}
Hence, we know that $\widehat E^N_\tau$ converge to $\widehat E_\tau$ in $L^1(\mathbb T)$.
We now recall that $\{\widehat E^N_\tau\}_{N \geq 1}$ are $M$-Lipschitz, which implies by Ascoli-Arzel\`a
that, up to subsequences, they converge uniformly to some limit,  but by uniqueness of the limit they have to converge uniformly to  $\widehat E_\tau$.
Thanks to this fact we finally obtain
$$
I_1 \leq \sup_{x \in \mathbb T}|\widehat E^N_\tau(x)-\widehat E_\tau(x)| \to 0,
$$
which concludes the proof.

\section{Proofs of Theorems \ref{thm1} and \ref{thm2}}
\label{sec:proofs}

Our aim is now to prove Theorems \ref{thm1} and \ref{thm2}. 
The principle is first to adapt some results from \cite{Gr96} for the (VPME) system in terms of the $W_1$ distance, which allows
us to settle the case where $h_{\e,0}=0$.
In a second time, we apply the stability estimate of Theorem \ref{thm4}.

\subsection{The fluid point of view and convergence for uniformly analytic initial data}
\label{subsec:grenier}
We describe in this section the approach introduced by Grenier in \cite{Gr96} for the study of the quasineutral limit for the classical Vlasov-Poisson system.
As we shall see, this can be adapted without difficulty to (VPME)$_\e$.

We assume that, for all $\e \in (0,1)$, $g_{0,\e}(x,v)$ is a \emph{continuous} function; following Grenier \cite{Gr96}, we write each initial condition as a ``superposition of Dirac masses in velocity'':
$$
g_{0, \e}(x,v) = \int_\M \rho_{0,\e}^\theta(x) \delta_{v= v_{0,\e}^\theta(x)} \, d\mu(\theta)
$$
with $\M:= \br$, $d\mu(\theta) = \frac{1}{\pi} \frac{d\theta}{1+\theta^2}$, 
$$\rho_{0, \e}^\theta= \pi (1+ \theta^2) g_{0,\e}(x,\theta), \quad  v_{0,\e}^\theta = \theta.$$
This leads to the study of the behavior as $\e\to0$ for solutions to the multi-fluid pressureless Euler-Poisson system 
\begin{equation}
\label{fluid}
 \left\{ \begin{array}{ccc}\pt_t \rho_\e^\theta+ \pt_x (\rho_\e^\theta v_\e^\theta)=0,  \\
\pt_t v_\e^\theta + v_\e^\theta \pt_x v_\e^\theta = E_\e, \\
E_\e=-U_\e', \\
\e^2 U_\e''=e^{U_\e}- \int_\M \rho_\e^\theta \, d\mu(\theta),\\
\rho_\e^\theta\vert_{t=0}= \rho_{0,\e}^\theta, v_\e^\theta\vert_{t=0}= v_{0,\e}^\theta.
\end{array} \right.
\end{equation}
One then checks that defining
$$g_\e(t,x,v) = \int_\M \rho_\e^\theta(t,x) \delta_{v= v_\e^\theta(t,x)} \, d\mu(\theta)$$
gives a weak solution to \eqref{vpme-quasi} (as an application of Theorem~\ref{thm4} and of the subsequent estimates, this is actually \emph{the unique} weak solution to \eqref{vpme-quasi} with initial datum $g_{0,\e}$).

The formal limit system, which is associated to the kinetic isothermal Euler system, is the following multi fluid isothermal Euler system:
\begin{equation}
\label{limit-fluid}
\left\{ \begin{array}{ccc}\pt_t \rho^\theta+ \pt_x (\rho^\theta v^\theta)=0,  \\
\pt_t v^\theta + v^\theta \pt_x v^\theta = E, \\
E=-U', \\
U= \log\left( \int_\M \rho^\theta \, d\mu(\theta)\right),\\
\rho^\theta\vert_{t=0}= \rho_{0}^\theta, v^\theta\vert_{t=0}= v_{0}^\theta,
\end{array} \right.
\end{equation}
where the $\rho_0^\theta$ are defined as the limits of $\rho_{0,\e}^\theta$ (which are thus supposed to exist)  and $v_0^\theta=\theta$. 

As before, one checks that defining
$$g(t,x,v) = \int_\M \rho^\theta(t,x) \delta_{v= v^\theta(t,x)} \, d\mu(\theta)$$
gives a weak solution to the kinetic Euler isothermal system.

Recalling the analytic norms used by Grenier in \cite{Gr96}
(see Definition \ref{def:norm}), we can adapt
the results of \cite[Theorems 1.1.2, 1.1.3 and Remark 1 p. 369]{Gr96} to get the following proposition. 

\begin{prop}
\label{grenier}
Assume that there exist $\delta_0,C,\eta>0$, with $\eta$ small enough,  such that
$$
\sup_{\e\in (0,1)}\sup_{v \in \br} (1+v^2) \| g_{0,\e} (\cdot,v)\|_{B_{\delta_0}}  \leq C,
$$
and that 
$$
\sup_{\e\in (0,1)} \left\|  \int_\br g_{0,\e}(\cdot,v) \, dv  -1 \right\|_{B_{\delta_0}} < \eta.
$$
%
Denote for all $\theta \in \br$,
$$\rho_{0, \e}^\theta= \pi (1+ \theta^2) g_{0,\e}(x,\theta), \quad  v_{0,\e}^\theta = v^\theta= \theta.$$
Assume that for all $\theta \in \br$, $\rho_{0,\e}^\theta$ has a limit in the sense of distributions and denote
$$
\rho_0^\theta= \lim_{\e \to 0} \rho_{0,\e}^\theta.  
$$
Then there exist $\delta_1>0$ and $T>0$ such that:
\begin{itemize}
\item for all $\e \in (0,1)$, there is a unique solution $(\rho_\e^\theta, v_\e^\theta)_{\theta \in M}$ of \eqref{fluid} with initial data $(\rho_{0,\e}^\theta, v_{0,\e}^\theta)_{\theta \in M}$, such that $\rho_\e^\theta, v_\e^\theta \in C([0,T]; B_{\delta_1})$ for all $\theta \in M$ and $\e \in (0,1)$, with bounds that are uniform in $\e$;
\item there is a unique solution $(\rho^\theta, v^\theta)_{\theta \in M}$ of \eqref{limit-fluid} with initial data $(\rho_{0}^\theta, v_{0}^\theta)_{\theta \in M}$, such that $\rho^\theta, v^\theta \in C([0,T]; B_{\delta_1})$
 for all $\theta \in M$;
\item for all $s \in \mathbb{N}$, we have
\be
\label{eq:conv}
\sup_{\theta \in M} \sup_{t \in [0,T]} \left[ \| \rho_\e^\theta - \rho^\theta\|_{H^s (\bt)} +  \| v_\e^\theta - v^\theta\|_{H^s (\bt)} \right] \to_{\e \to 0 } 0.
\ee

\end{itemize}

\end{prop}

Remark that analyticity is actually needed only in the position variable, and not in the velocity variable.
This allows us, for instance, to consider initial data which are compactly supported in velocity.

We shall not give a complete proof of this result (which is of Cauchy-Kovalevski type), since it is very close to the one given by Grenier in \cite{Gr96} for the classical Vlasov-Poisson system, but we just emphasize the main differences.

First of all we begin by noticing that one difficulty in the classical case comes from the fact the one can not directly use the Poisson equation
$$-\e^2  U''_\e = \rho_\e - 1$$
if one wants some useful uniform analytic estimates for the electric field. Because of this issue, a combination of the Vlasov and Poisson equation is used in \cite{Gr96}, which allows one to get a kind of wave equation solved by $U_\e$.
This shows in particular that the electric field has a highly oscillatory behavior in time (the fast oscillations in time correspond to the so-called plasma waves) which have to be filtered in order to obtain convergence. For this reason, Grenier needs to introduce some correctors in order to get convergence of the velocity fields
(these oscillations and correctors vanish only if the initial conditions are well-prepared, \emph{i.e.} verify some compatibility conditions).

For the (VPME) system, that is when one adds the exponential term in the Poisson equation, such a problem does not occur. To explain this, consider first the linearized Poisson equation
$$-\e^2  U''_\e + U_\e =  \int_\M \rho_\e^\theta \, d\mu(\theta) - 1$$
and observe that this equation is appropriate to get uniform analytic estimates. Indeed, writing
$$
U_\e = (Id - \e^2 \pt_{xx})^{-1} \left(  \int_\M \rho_\e^\theta \, d\mu(\theta)- 1\right),
$$
this shows that if $\rho_\e$ is analytic then also $U_\e$ (and so $E_\e$)
is analytic, which implies that there are no fast oscillations in time, contrary to the classical case.
In particular, our convergence result holds without the need of adding any correctors.

A second difference concerns the existence of analytic solutions on an interval of time $[0,T]$
independent of $\e$: the construction of Grenier of analytic solutions
is based on a Cauchy-Kovalevski type proof based on an iteration procedure in a scale of Banach spaces
(see \cite[Section 2.1]{Gr96}). Most of the estimates used to prove that such iteration converge 
use the Fourier transform, that is unavailable in our case since the Poisson equation
$$-\e^2  U''_\e + e^{U_\e} =  \int_\M \rho_\e^\theta \, d\mu(\theta)$$
is nonlinear.
However, since we deal with analytic functions, we can express everything in power series 
to use the Fourier transform and obtain some a priori estimates in the analytic norm. Furthermore, one can write the Poisson equation as
$$-\e^2  U''_\e + U_\e =  \int_\M \rho_\e^\theta \, d\mu(\theta) - 1 - (e^{U_\e} -  U_\e -1)$$
and rely on the fact that the ``error term'' $(e^{U_\e} -  U_\e -1)$ is quadratic in $U_\e$ (which is expected to be small in the regime where $\sup_{\e\in (0,1)} \left\|  \int_\M \rho_\e^\theta \, d\mu(\theta)   -1 \right\|_{B_{\delta_0}} \ll 1$), and thus can be handled in the approximation scheme used in \cite{Gr96}.

\bigskip

We deduce the next corollary.
\begin{cor}
\label{cor:1}
With the same assumptions and notation as in Proposition \ref{grenier}, we have
\be
\label{W1-0}
\sup_{t \in [0,T]} W_1(g_\e (t), g(t)) \to_{\e\to0} 0,
\ee
where
\be
\label{g1}
g_\e(t,x,v) = \int_\M \rho_\e^\theta(t,x) \delta_{v= v_\e^\theta(t,x)} \, d\mu(\theta), \qquad g(t,x,v) = \int_\M \rho^\theta(t,x) \delta_{v= v^\theta(t,x)} \, d\mu(\theta).
\ee
\end{cor}

\begin{proof}
The convergence \eqref{W1-0} follows from \eqref{eq:conv}, and the Sobolev embedding theorem. We have indeed for all $t\in [0,T]$:
\begin{align*}
&W_1(g_\e (t), g(t)) = \sup_{\|\varphi\|_{\text{Lip}} \leq 1} \langle g_\e -g, \, \varphi \rangle \\
&=  \sup_{\|\varphi\|_{\text{Lip}} \leq 1} \left\{ \int_{\bt }  \int_\M (\rho_\e^\theta(t,x)\varphi(x,v_\e^\theta(t,x)) - \rho^\theta(t,x)\varphi(x,v^\theta(t,x))) \, d\mu(\theta)  \,  dx \right\} \\
&=\sup_{\|\varphi\|_{\text{Lip}} \leq 1} \left\{ \int_{\bt }  \int_\M \rho_\e^\theta(t,x)(\varphi(x,v_\e^\theta(t,x))-\varphi(x,v^\theta(t,x))  \, d\mu(\theta)  \,  dx \right\}  \\
&+ \sup_{\|\varphi\|_{\text{Lip}} \leq 1} \left\{ \int_{\bt}  \int_\M (\rho_\e^\theta(t,x) - \rho^\theta(t,x))\varphi(x,v^\theta(t,x)) \, d\mu(\theta)  \,  dx \right\}.
\end{align*}
Thus, we deduce the estimate
\begin{align*}
W_1(g_\e (t), g(t)) 
&\leq \sup_{\|\varphi\|_{\text{Lip}} \leq 1}  \sup_{\e \in (0,1), \, \theta \in M} \|\rho_\e^\theta\|_{\infty} \|\varphi\|_{\text{Lip}} \int_\M \| v_\e^\theta(t,x)-v^\theta(t,x)\|_\infty  \, d\mu(\theta) \\
&+ \sup_{\|\varphi\|_{\text{Lip}} \leq 1}  \int_\M \|\rho_\e^\theta- \rho_\theta\|_{\infty} \, d\mu(\theta) \|\varphi\|_{\text{Lip}}  \left(1/2+ \sup_{\theta \in M} \|v^\theta(t,x)\|_\infty\right) \\
&+  \sup_{\|\varphi\|_{\text{Lip}} \leq 1} \left\{ \int_{\bt \times \br}  \int_\M (\rho_\e^\theta(t,x) - \rho^\theta(t,x))\varphi(0,0) \, d\mu(\theta)  \,  dx \right\}.
\end{align*}
We notice that the last term is equal to $0$ since for all $t \geq 0$,
$$
\int_\bt \int_\M \rho_\e^\theta(t,x) \, d\mu(\theta) \, dx =\int_\bt \int_\M \rho_\e^\theta(t,x) \, d\mu(\theta) \, dx =1,
$$
by conservation of the total mass. After taking the supremum in time, we also see that the other two terms converge to $0$, using the $L^\infty$ convergence of $(\rho_\e^\theta, v_\e^\theta)$ to $(\rho_{0}^\theta, v_{0}^\theta)$. 
This concludes the proof.
\end{proof}

\bigskip

This approach is also relevant for singular initial data such as the sum of Dirac masses in velocity:
$$
g_{0,\e}(x,v) = \sum_{i=1}^N \rho_{0,\e}^i(x) \delta_{v=v_{0,\e}^i(x)}.
$$
and we have a similar theorem assuming that  $(\rho_{0,\e}^i, v_{0,\e}^i)$ is uniformly analytic.

In this case $\M= \{1,\cdots,N\}$ and $d\mu$ is the counting measure. This leads to the study of the behavior as $\e \to 0$ of the system (for $i \in  \{1,\cdots,N\}$)
\begin{equation}
\label{eq:Ndirac}
\left\{ \begin{array}{ccc}\pt_t \rho^i_\e+ \pt_x (\rho^i_\e v^i_\e)=0,  \\
\pt_t v^i_\e+ v^i_\e \pt_x v^i_\e = E_\e, \\
E_\e=-U'_\e, \\
\e^2 U''_\e = e^{U_\e}- \left(  \sum_{i=1}^N  \rho^i_\e \right),\\
\rho^i_\e\vert_{t=0}= \rho_{0,\e}^i, v^i_\e\vert_{t=0}= v_{0,\e}^i.
\end{array} \right.
\end{equation}
and the formal limit is the following multi fluid isothermal system
\begin{equation}
\label{eq:Ndiraclimit}
\left\{ \begin{array}{ccc}\pt_t \rho^i+ \pt_x (\rho^i v^i)=0,  \\
\pt_t v^i + v^i \pt_x v^i = E, \\
E=-U', \\
U= \log\left(  \sum_{i=1}^N  \rho^i \right),\\
\rho^i\vert_{t=0}= \rho_{0}^i, v^i\vert_{t=0}= v_{0}^i.
\end{array} \right.
\end{equation}

As before,  adapting the arguments in \cite{Gr96}, we obtain the following proposition and its corollary.
\begin{prop}
\label{grenier2}
Assume that there exist $\delta_0,C,\eta>0$, with $\eta$ small enough, such that
$$
\sup_{\e\in (0,1)}\sup_{i \in \{ 1, \cdots , N\}} \| \rho^i_{0,\e}\|_{B_{\delta_0}} +  \| v^i_{0,\e}\|_{B_{\delta_0}}   \leq C,
$$
and that 
$$
\sup_{\e\in (0,1)} \left\|  \sum_{i=1}^N  \rho^i_{0,\e}   -1 \right\|_{B_{\delta_0}} < \eta.
$$
%
Assume that for all $i=1,\cdots, N$, $\rho^i_{0,\e}, v^i_{0,\e}$ admit a limit in the sense of distributions and denote
$$
\rho^i_{0} =\lim_{\e \to 0} \rho^i_{0,\e}, \quad v^i_{0} = \lim_{\e \to 0} v^i_{0,\e}.
$$
Then there exist $\delta_1>0$ and $T>0$ such that:
\begin{itemize}
\item for all $\e \in (0,1)$, there is a unique solution $(\rho_\e^i, v_\e^i)_{i \in \{ 1, \cdots , N\}}$ of \eqref{eq:Ndirac} with initial data $(\rho_{0,\e}^i, v_{0,\e}^i)_{i \in \{ 1, \cdots , N\}}$, such that $\rho_\e^i, v_\e^i \in C([0,T]; B_{\delta_1})$  for all $i \in \{ 1, \cdots , N\}$ and $\e \in (0,1)$, with bounds that are uniform in $\e$;
\item there is a unique solution $(\rho^i, v^i)_{i \in \{ 1, \cdots , N\}}$ of \eqref{eq:Ndiraclimit} with initial data $(\rho_{0}^i, v_{0}^i)_{i \in \{ 1, \cdots , N\}}$, such that  $\rho^i, v^i \in C([0,T]; B_{\delta_1})$
for all $i \in \{ 1, \cdots , N\}$;
\item for all $s \in \mathbb{N}$, we have
$$
\sup_{i \in \{ 1, \cdots , N\}} \sup_{t \in [0,T]} \left[ \| \rho_\e^i - \rho^i\|_{H^s (\bt)} +  \| v_\e^i - v^i\|_{H^s (\bt)} \right] \to_{\e \to 0 } 0.
$$
\end{itemize}

\end{prop}

\begin{cor}
\label{cor:2.1}
With the same assumptions and notation as in the Proposition \ref{grenier2}, for all $t  \in [0,T]$
we have
\be
\label{W1-0-again}
W_1(g_\e (t), g(t)) \to_{\e\to0} 0,
\ee
where
\be
\label{g2}
g_\e(t,x,v) = \sum _{i \in \{ 1, \cdots , N\}}  \rho_\e^i(t,x) \delta_{v= v_\e^i(t,x)}, \quad g(t,x,v) = \sum_{i \in \{ 1, \cdots , N\}}  \rho^i(t,x) \delta_{v= v^i(t,x)}.
\ee
\end{cor}

\subsection{End of the proof of Theorem \ref{thm1} and Theorem \ref{thm2}}

We are now in position to conclude. 
Let $(f_\e)$ a sequence of global weak solutions to \eqref{vpme-quasi} with initial conditions $(f_{0,\e})$ (obtained thanks to Theorem~\ref{thm-exi}).

We denote by $(g_\e)$ the sequence of weak solutions to \eqref{vpme-quasi} with initial conditions $(g_{0,\e})$, defined by \eqref{g1} for the case of Theorem~\ref{thm1} and \eqref{g2} for the case of Theorem~\ref{thm2}. Using the triangle inequality, we have
\begin{equation*}
W_1(f_\e(t), g(t)) \leq W_1(f_\e(t), g_\e (t)) + W_1(g_\e(t), g(t)),
\end{equation*}
where $g$ is defined by \eqref{g1} for the case of Theorem~\ref{thm1} and \eqref{g2} for the case of Theorem~\ref{thm2}.

For the first term, we use Theorem \ref{thm4} to get
\begin{align*}
W_1(f_\e(t), g_\e (t)) &\leq W_1(g_{0,\e} + h_{0,\e}, g_{0,\e} )   \frac{1}{\e} e^{\frac1\e \left[(1 +3/\e^2 )t +(8+ \frac{1}{\e^2} e^{15/(2\e^2)} \Black) \int_0^t\| \rho_\e(\tau)\|_\infty\,d\tau\right]}  \\
&= W_1(h_{0,\e}, 0)   \frac{1}{\e} e^{\frac1\e \left[(1 +3/\e^2 )t +(8+ \frac{1}{\e^2} e^{15/(2\e^2)} \Black) \int_0^t\| \rho_\e(\tau)\|_\infty\,d\tau\right]},
\end{align*}
where $\rho_\e$ is here the local density associated to $g_\e$. By Proposition \ref{grenier} (for the case of Theorem \ref{thm1}) and Proposition \ref{grenier2} (for the case of Theorem \ref{thm2}), there exists $C_0>0$ such that for all $\e \in (0,1)$, 
$$
\sup_{\tau \in [0,T]}  \| \rho_\e(\tau)\|_\infty \leq C_0.
$$
Consequently,  we observe that taking
$$
\varphi(\e)= \frac{1}{\e} \exp \left( \frac{\lambda}{\e^3}  \exp \frac{15}{2 \e^2}\right),
$$
with $\lambda<0$, we have, by assumption on $h_{0,\e}$ (take a smaller $T$ if necessary) that
$$
\sup_{t \in [0,T]} W_1(f_\e(t), g_\e (t))  \to_{\e \to 0} 0.
$$
We also get that $W_1(g_\e(t), g(t))$ converges to $0$, applying Corollary \ref{cor:1} for the case of Theorem \ref{thm1}, and Corollary \ref{cor:2.1} for the case of Theorem \ref{thm2}.

This concludes the proofs of Theorems \ref{thm1} and \ref{thm2}.

\section{Proof of Proposition \ref{thm3}}
\label{sec:insta}

We now discuss a non-derivation result which was first stated by Grenier in the note \cite{Gr99}, and then studied in more details  by the first author and Hauray in \cite{HKH} (the latter was stated for a general class of homogeneous data, i.e., independent of the position).
In \cite{HKH} such results are given either for the classical Vlasov-Poisson system or for the linearized (VPME) system, but the proofs can be adapted to (VPME).

We first recall two definitions from \cite{HKH}.

\begin{defi} 
\label{def:Penmod}
We say that a homogeneous profile $\mu(v)$ with $\int \mu \, dv=1$ satisfies the Penrose instability 
criterion if there exists a local minimum point $\bar v$ of $\mu$ such that the following inequality holds:
\begin{equation} \label{def:Pen-first-mod}
\int_{\br} \frac{ \mu(v) - \mu(\bar v)} {(v-\bar v)^2} \, dv >   1.
\end{equation}
If the local minimum is flat, i.e., is reached on an interval $[\bar v_1, \bar v_2]$, then~\eqref{def:Pen-first-mod} has to be satisfied for all $\bar v \in [\bar v_1, \bar v_2]$.  
\end{defi}

\begin{defi}
\label{cond-pos}
 We say that a positive and  $C^1$ profile $\mu(v)$ satisfies the $\delta$-condition\footnote{The appellation is taken from \cite{HKH}. }
if
\begin{equation} \label{cond:alpha}
\sup_{v \in \br}  \,  \frac{|\mu'(v)|}{(1+ |v|)\mu(v)} < + \infty.
\end{equation}
\end{defi}

We can now state the theorem taken from \cite{HKH}.
\begin{thm} 
\label{thmGrenier-revisited}
Let $\mu(v)$ be a smooth profile satisfying the Penrose instability criterion.  Assume that $\mu$ is positive and satisfies the $\delta$-condition\footnote{It is also possible to consider a non-negative $\mu$ but the relevant condition is rather involved, we refer to \cite{HKH} for details.}. For any $N>0$ and $s>0$, there exists a 
sequence of non-negative initial data $(f_{0,\e})$ such that 
$$
\| f_{\e,0}- \mu\|_{W^{s,1}_{x,v}} \leq \e^N,
$$
and denoting by $(f_\e)$ the sequence of solutions
to \eqref{vpme-quasi} with initial data $(f_{0,\e})$, the following holds:

\begin{enumerate}
\item {\bf $L^1$ instability for the macroscopic observables:} consider the density $\rho_\e := \int f_\e \, dv$ and the electric field $E_\e = - \pt_x U_\e$. For all $\alpha\in [0,1)$, we have
\begin{equation}
\label{insta:macro}
\liminf_{\e \rightarrow 0} \sup_{t \in [0,\e^\alpha]} \left\| \rho_\e(t) - 1 \right\|_{L^1_{x}} >0,
\qquad
\liminf_{\e \rightarrow 0} \sup_{t \in [0,\e^\alpha]} {\e}\left\| E_\e \right\|_{L^1_x} >0.
\end{equation}

\item {\bf Full instability for the distribution function:} for any $r \in \bz$, we have
\begin{equation}
\label{insta:full}
\liminf_{\e \rightarrow 0} \sup_{t \in [0,\e^\alpha]} \left\| f_\e(t) - \mu \right\|_{W^{r,1}_{x,v}} >0.
\end{equation}
\end{enumerate}
\end{thm}

We deduce a proof of Proposition \ref{thm3} from this result. Indeed, take a smooth $\mu$ satisfying the assumptions of Theorem \ref{thmGrenier-revisited}, and consider the sequence of initial conditions $(f_{0,\e})$ given by this theorem.

By the Sobolev imbedding theorem in dimension $1$, the space $W^{2,1}(\bt \times \br)$ is continuously imbedded in the space $W^{1,\infty}(\bt \times \br)$ (i.e., bounded Lipschitz functions), hence  there exists a constant $C>0$ such that, for all $\e \in (0,1)$,
\begin{align*}
W_1(f_\e, \mu) &=  \sup_{\|\varphi\|_{\text{Lip}} \leq 1} \langle f_\e -\mu, \, \varphi \rangle  \\
&\geq \sup_{\|\varphi\|_{W^{2,1}} \leq C} \langle f_\e -\mu, \, \varphi \rangle \\
&= C \|f_\e - \mu\|_{W^{-2,1}}.
\end{align*}
Therefore, by \eqref{insta:full} with $r=-2$, we deduce that 
$$
\liminf_{\e \rightarrow 0} \sup_{t \in [0,\e^\alpha]} W_1(f_\e, \mu) >0,
$$
which proves the claimed result.

\section{Proof of Proposition \ref{thm3-2}}
\label{sec:sta}

As we already mentioned in the introduction, in the case of one single Dirac mass in velocity, the situation is much more favorable. This was first shown by Brenier in \cite{Br00} for the quasineutral limit of the classical Vlasov-Poisson system, using the so-called relative entropy (or modulated energy) method. It was then adapted by the first author in \cite{HK} for the quasineutral limit of (VPME).

In this case, the expected limit is the Dirac mass in velocity
$$
f(t,x,v)= \rho(t,x,v) \delta_{v= u(t,x)}
$$
which is a weak solution of \eqref{formal} whenever $(\rho, u)$ is a strong solution to the isothermal Euler system
\begin{equation}
\label{IE}
\left\{ \begin{array}{ccc}\pt_t \rho+ \pt_x (\rho u)=0,  \\
\pt_t u + u \pt_x u + \frac{\pt_x  \rho}{\rho} = 0, \\
\rho\vert_{t=0}= \rho_{0}, v\vert_{t=0}= u_{0}.
\end{array} \right.
\end{equation}
This is a hyperbolic and symmetric system, that admits local smooth solutions for smooth initial data (in this section, smooth means $H^s$ with $s$ larger than $2$).
From \cite{HK} we deduce the following stability result\footnote{In \cite{HK}, computations are done for the model posed on $\br^3$, but the same holds for the model set on $\bt$.}.
\begin{thm}
\label{relative}
Let $\rho_0>0, u_0$ be some smooth initial conditions for \eqref{IE}, and $\rho, u$ the associated strong solutions of \eqref{IE} defined on some interval of time $[0,T]$, where $T>0$.
Let $f_\e$ be a non-negative global weak solution of~\eqref{vpme-quasi} such that $f_\e \in L^1 \cap L^\infty$, $\int f_\e \, dx dv=1$,  and with uniformly bounded energy, i.e., there exists $A>0$, such that for all $\e \in (0,1)$,
$$
\mathcal{E}_\e(t):= \frac{1}{2}\int f_\e\vert v\vert ^2 dv \,dx + \int \left(e^{U_\e} \log  e^{U_\e}   - e^{U_\e} +1\right)\,dx + \frac{\e^2}{2} \int \vert U'_\e \vert^2 dx \leq A.
$$
For all $\e \in (0,1)$, define the relative entropy
$$
\mathcal{H}_\e(t):=\frac{1}{2}\int f_\e\vert v - u\vert ^2 dv \,dx + \int \left(e^{U_\e} \log \left( e^{U_\e}/ \rho \right)  - e^{U_\e} +\rho\right)\,dx + \frac{\e^2}{2} \int \vert U'_\e \vert^2 dx.
$$
Then there exists $C>0$ and a function $G_\e(t)$ satisfying $\| G_\e\|_{L^\infty([0,T])} \leq C \e$ such that, for all $t \in [0,T]$, 
$$
{\mathcal{H}}_\e(t) \leq   {\mathcal{H}}_\e(0) +  G_\e(t) +  C\int_0^t \Vert \pt_x u \Vert_{L^\infty} \mathcal{H}_\e(s) ds.
$$
In particular, if $ {\mathcal{H}}_\e(0)   \to_{\e \to 0} 0$, then $ {\mathcal{H}}_\e(t)   \to_{\e \to 0} 0$
 for all $t \in [0,T]$.

In addition, if there is $C_0>0$ such that ${\mathcal{H}}_\e(0)  \leq C_0 \e$, then there is $C_T>0$ such that $ {\mathcal{H}}_\e(t)  \leq C_T \e$ for all $t \in [0,T]$ and $\e \in (0,1)$.
\end{thm}

Notice that, by a convexity argument, one also deduces that $\rho_\e= \int f_\e \, dv \rightharpoonup \rho$ (and $e^{U_\e} \rightharpoonup \rho$ as well) and $j_\e = \int f_\e v \, dv\rightharpoonup \rho u$ in a weak-$\star$ sense (see \cite{HK}).

\bigskip

We can actually deduce the following corollary, which is a precise version of Proposition \ref{thm3-2}.
\begin{cor}
\label{cor:2.2}
With the same assumptions and notation as in the previous theorem, the following convergence results hold:
\begin{enumerate}
\item If $ {\mathcal{H}}_\e(0)   \to_{\e \to 0} 0$, then
$$
\sup_{t \in [0,T]} W_1(f_\e, \rho\, \delta_{v= u}) \to_{\e \to 0} 0.
$$
\item If ${\mathcal{H}}_\e(0)  \leq C_0 \e$, then there is $C_T'>0$ such that, for all $\e \in (0,1)$,
$$
\sup_{t \in [0,T]} W_1(f_\e, \rho \,\delta_{v= u}) \leq C'_T \sqrt{\e}.
$$
\end{enumerate}
\end{cor}

\begin{proof}
Recall that we denote $\rho_\e = \int f_\e \, dv$.
Let $\varphi$ such that $\|\varphi\|_{\text{Lip}} \leq 1$ and compute
\begin{align*}
\langle f_\e -  \rho\, \delta_{v= u}, \,  \varphi \rangle &= \langle f_\e -  \rho_\e\,  \delta_{v= u}, \,  \varphi \rangle + \langle (\rho_\e -  \rho )\,\delta_{v= u}, \,  \varphi \rangle \\
&=: A_1 + A_2.
\end{align*}
Using the bound $\|\varphi\|_{\text{Lip}} \leq 1$, the Cauchy-Schwarz inequality, the fact that $f_\e$ is non-negative and of total mass $1$, and the definition of $\mathcal{H}_\e(t)$, we have
\begin{align*}
|A_1| &= \left|\int_{\bt\times\br} f_\e(t,x,v) (\varphi(t,x,v)- \varphi(t,x,u(t,x)) \, dv dx \right|\\
&\leq \int_{\bt\times\br} f_\e(t,x,v) |\varphi(t,x,v)- \varphi(t,x,u(t,x))| \, dv dx \\
&\leq \int_{\bt\times\br} f_\e(t,x,v) |v- u(t,x)| \, dv dx \\
&\leq  \left(\int_{\bt\times\br} f_\e(t,x,v) |v- u(t,x)|^2 \, dv dx\right)^{1/2}  \left(\int_{\bt\times\br} f_\e(t,x,v)  \, dv dx\right)^{1/2} \\
&\leq \sqrt{2} \sqrt{\mathcal{H}_\e(t)}.
\end{align*}
Considering $A_2$, we first have
\begin{align*}
A_2 &= \int_{\bt} (\rho_\e(t,x)- \rho(t,x)) \varphi(x,u(t,x)) \, dx \\
&=  \int_{\bt} (\rho_\e(t,x)- \rho(t,x)) \big(\varphi(x,u(t,x)) - \varphi(0,0) \big) \, dx
\end{align*}
since the total mass is preserved (and equal to $1$). Furthermore, we use the Poisson equation
$$
 \rho_\e = e^{U_\e}- \e^2 U_\e'',
$$
 to rewrite $A_2$ as
 \begin{align*}
A_2 &=   \int_{\bt} (e^{U_\e}- \rho(t,x)) \left(\varphi(x,u(t,x)) - \varphi(0,0) \right) \, dx  \\
&\qquad- \e^2  \int U_\e''  \left(\varphi(x,u(t,x)) - \varphi(0,0) \right) \, dx \\
&=: A_2^1 + A_2^2.
\end{align*}
Let us start with $A_2^2$. By integration by parts, the Cauchy-Schwarz inequality, and the bound $\|\varphi\|_{\text{Lip}} \leq 1$, we have
\begin{align*}
|A_2^2| &= \e^2 \left| \int_{\bt} U_\e' [\pt_x \varphi(x,u(t,x)) + \pt_x  u(t,x) \pt_ v \varphi(x,u(t,x))]  \, dx \right| \\
&\leq \e [1+ \| \pt_x u\|_{\infty}]  \left( \e^2 \int_{\bt}  |U_\e'|^2  \, dx \right)^{1/2} \\
&\leq \e [1+ \| \pt_x u\|_{\infty}] \sqrt{\mathcal{E}_\e(t)} \\
&\leq \sqrt{A} [1+ \| \pt_x u\|_{\infty}] \e.
\end{align*}
For $A_2^1$, we shall use the classical inequality
$$
(\sqrt{y}- \sqrt{x})^2 \leq x \log (x/y) -x +y,
$$
for $x,y >0$, and proceed as follows:
\begin{align*}
|A_2^1| &= \left| \int_{\bt} \left(e^{U_\e}- \rho(t,x)\right) \left(\varphi(x,u(t,x)) - \varphi(0,0) \right) \, dx  \right| \\
&= \left| \int_{\bt}\left(e^{\frac{1}{2} U_\e}- \sqrt{\rho(t,x)}\right)\left(e^{\frac{1}{2} U_\e}+ \sqrt{\rho(t,x)}\right) \left(\varphi(x,u(t,x)) - \varphi(0,0) \right) \, dx  \right| \\
&\leq \left( \int_{\bt} \left(e^{\frac{1}{2} U_\e}- \sqrt{\rho(t,x)}\right)^2\left|\varphi(x,u(t,x)) - \varphi(0,0) \right| \, dx  \right)^{1/2} \\
&\qquad \times \left( \int_{\bt} \left(e^{\frac{1}{2} U_\e}+ \sqrt{\rho(t,x)}\right)^2\left|\varphi(x,u(t,x)) - \varphi(0,0) \right| \, dx  \right)^{1/2}. 
\end{align*}
We have
\begin{align*}
&\left( \int_{\bt} \left(e^{\frac{1}{2} U_\e}- \sqrt{\rho(t,x)}\right)^2\left|\varphi(x,u(t,x)) - \varphi(0,0) \right| \, dx  \right)\\
&\qquad \leq (1+ \|u\|_\infty) \int \left(e^{U_\e} \log \left( e^{U_\e}/ \rho \right)  - e^{U_\e} +\rho\right)\,dx \\
&\qquad  \leq (1+ \|u\|_\infty) \mathcal{H}_\e(t),
\end{align*}
and likewise we obtain the rough bound
\begin{align*}
&\left( \int_{\bt} \left(e^{\frac{1}{2} U_\e}+ \sqrt{\rho(t,x)}\right)^2\left|\varphi(x,u(t,x)) - \varphi(0,0) \right| \, dx  \right) \\
&\qquad \leq 2 \int_{\bt} \left(e^{\frac{1}{2} U_\e}- 1\right)^2\left|\varphi(x,u(t,x)) - \varphi(0,0) \right| \, dx   \\
&\qquad + 2 \int_{\bt} \left(\sqrt{\rho(t,x)}+1\right)^2\left|\varphi(x,u(t,x)) - \varphi(0,0) \right| \, dx    \\
&\qquad \leq 2 (1+ \|u\|_\infty)\left( \mathcal{E}_\e(t) +  \int_{\bt} \left(\sqrt{\rho(t,x)}+1\right)^2 \, dx\right) \\
 &\qquad \leq 2 (1+ \|u\|_\infty)\left( A +  \int_{\bt} \left(\sqrt{\rho(t,x)}+1\right)^2 \, dx\right).
\end{align*}
As a consequence, we get
\begin{equation*}
|A_2^2| \leq \sqrt{2} (1+\|u\|_\infty)\left( A +  \int_{\bt} \left(\sqrt{\rho(t,x)}-1\right)^2 \, dx\right)^{1/2} \sqrt{\mathcal{H}_\e(t)}.
\end{equation*}
Gathering all pieces together, we have shown
\begin{multline*}
\langle f_\e -  \rho \delta_{v= u}, \,  \varphi \rangle  \\
\leq 
 \sqrt{2} \left[1+ (1+\|u\|_\infty)\left( A +  \int_{\bt} \left(\sqrt{\rho(t,x)}+1\right)^2 \, dx\right)^{1/2}\right] \sqrt{\mathcal{H}_\e(t)} 
 + \sqrt{A} [1+ \| \pt_x u\|_{\infty}] \e,
\end{multline*}
which allows us to conclude the proof applying Theorem \ref{relative}.
\end{proof}


\bigskip

{\bf Acknowledgements.} The authors would like to thank Maxime Hauray for some helpful discussions about his paper \cite{HauX}.  We are also grateful to the anonymous referees for their remarks about this work, and for helping us to correct several typos. 
\bibliographystyle{plain}
\bibliography{weak-strong}

\end{document}